\documentclass[11pt,letterpaper]{amsart}

\usepackage{mathrsfs}
\usepackage{cite}
\usepackage{color}
\usepackage{amssymb}
\theoremstyle{plain}

\newtheorem{lem}{Lemma A.}

\author[Charles Collot]{Charles Collot}
\address{Charles Collot, Laboratoire AGM, CY Cergy Paris Universit\'e, 2 avenue Adolphe Chauvin, 95300 Pontoise, France}
\email{ccollot@cyu.fr}

\author[K. Zhang]{Kaiqiang Zhang}
\address{Kaiqiang Zhang, Department of Mathematics, University of Science and Technology Beijing, Beijing, 30 Xueyuan Road, Haidian District, Beijing 100083, China}
\email{mathzhangkq@gmail.com}


\begin{document}

\numberwithin{equation}{section}
\renewcommand{\theequation}{\arabic{section}.\arabic{equation}}
\theoremstyle{plain}
\newtheorem{exam}{Example}[section]
\newtheorem{theorem}[exam]{Theorem}
\newtheorem{lemma}[exam]{Lemma}
\newtheorem{remark}[exam]{Remark}
\newtheorem{hyp}[exam]{Hypothesis}
\newtheorem{proposition}[exam]{Proposition}
\newtheorem{definition}[exam]{Definition}
\newtheorem{corollary}[exam]{Corollary}
\newtheorem{analytic}[exam]{Analytic extension principle}
\newtheorem{notation}[exam]{Notation}
\setlength{\baselineskip}{1.5\baselineskip}

\title[Keller-Segel system]{{On the stability of Type I self-similar blowups for the Keller-Segel system in three dimensions and higher}}
\begin{abstract}
We consider the parabolic-elliptic Keller-Segel system in spatial dimensions $d\geq3$, which corresponds to the mass supercritical case. Some solutions become singular in finite time, an important example being backward self-similar solutions. Herrero et al. and Brenner et al. showed the existence of such profiles, countably many in dimensions $3\leq d \leq 9$ and at least two for $d\geq 10$. We establish that all these self-similar profiles are stable along a set of initial data with finite Lipschitz codimension equal to the number of instable eigenmodes. This extends the recent finding of Glogi\'c et al. showing the stability of the fundamental self-similar profile. We obtain additional results, such as the possibility of the solutions we construct to originate from smooth and compactly supported initial data, their convergence at blow-up time, and the Lipschitz regularity of the blow-up time. Our proof extends the approach proposed in Collot et al., based on renormalizing the solution around a modulated self-similar solution, and using a spectral gap for the linearized operator in the parabolic neighbourhood of the singularity.
\end{abstract}

\maketitle

\section{Introduction}

\subsection{The parabolic-elliptic Keller-Segel system}

The present paper deals with the parabolic-elliptic Keller-Segel system
\begin{equation}\label{1.1}
\left\{\begin{array}{l}
\partial_{t} u=\Delta u-\nabla \cdot\left(u \nabla \Phi_{u}\right),  \\
0=\Delta \Phi_{u}+u,
\end{array}\quad \text{in}\ \mathbb{R}^{d} .\right.
\end{equation}
This system arises in describing chemotactic aggregation in biology \cite{Keller-Segel}. Here, $u(x,t)$ represents the cell density, while $\Phi_{u}$ stands for the concentration of chemoattractant. Additionally, system \eqref{1.1}  models the self-gravitating matter in stellar
dynamics \cite{Wolansky}, with $u(x,t)$ denoting the gas density and $\Phi_{u}$ representing the gravitational potential. We refer to \cite{Horstmann,Biler-2019} for further background on \eqref{1.1}.

Local well-posedness results for \eqref{1.1} can be found in \cite{Bedrossian-Masmoudi-2014,Biler-2018}. For $u$ radial, there holds
$$
\nabla \Phi_u(x)=-\frac{x}{ |\mathbb S^{d-1}| |x|^d}\int_{|y|<|x|}u(y)dy
$$
so that the system \eqref{1.1} becomes the radial nonlocal semilinear parabolic equation
$$
\partial_t u =\Delta u +\nabla.\left(u \ \frac{x}{ |\mathbb S^{d-1}| |x|^d} \int_{|y|<|x|}u(y)dy \right).
$$
Given a radial $u_0\in L^\infty(\mathbb{R}^d)$, there exists a maximal time of existence $T>0$ and a unique solution of \eqref{1.1} that belongs to $L^\infty([0,T']\times \mathbb R^d)$ for any $0<T'<T$, which is moreover smooth on $(0,T)\times \mathbb R^d$, see \cite{Giga-Mizoguchi-Senba-2011}. The solution blow up in finite time $T<\infty$ if
$$
\limsup_{t\to T^-}||u(t)||_{L^\infty(\mathbb{R}^d)}= \infty,
$$
and by a comparison argument, there holds the lower bound of the blow up rate
$$
||u(t)||_{L^\infty(\mathbb{R}^d)}\ge (T-t)^{-1}.
$$
One can find other lower bounds in \cite{Kozono-2010}.
If we discard the diffusion and transport in the first equation of \eqref{1.1}, one has $u_t=u^2$.
The solution of the corresponding ODE $y'=y^2$ when $y(0)>0$ is given by $y(t)=(T-t)^{-1}$, for $0<t<T=y(0)^{-1}$. It is said that the blow-up is of type I if a solution $u$ of \eqref{1.1} blows up at the ODE rate, i.e.
\begin{equation}\label{mar}
\limsup_{t\to T^-} \ (T-t) \| u(t)\|_{L^\infty(\mathbb R^d)}<\infty
\end{equation}
 whereas blow-up is said to be of type II if \eqref{mar} fails.

Solutions admit a scaling invariance, i.e., if $u$ is a solution of \eqref{1.1}, then so is
\begin{equation}\label{scaling}
u_\lambda(x,t)=\frac{1}{\lambda^2}u\bigg(\frac{x}{\lambda},\frac{t}{\lambda^2}\bigg)
\end{equation}
for all $\lambda>0$. The system \eqref{1.1} preserves mass, that is, if $u_0\in L^1$, then for all $t\in [0,T)$
$$
M:=\int_{\mathbb{R}^2}u_0(x)dx=\int_{\mathbb{R}^2}u(x,t)dx.
$$

\subsection{The two-dimensional case}

It has been proved in \cite{Nagai-1995,Childress-Jerome} that the solution exists globally for $d=1$. The case $d=2$ is called mass critical because the scaling transformation \eqref{scaling} preserves the total mass $M(u)=M(u_\lambda)$. Any blowup solution is necessarily of type II \cite{Naito-2008,Suzuki-2011}. A mass threshold of $8\pi$ between global existence and finite blow-up for localised solutions was conjectured in \cite{Childress,Childress-Jerome}, and later proven in \cite{Blanchet-2012,Diaz-1998,Blanchet-2006}. Namely, assume $u_0\geq 0$ with $\langle x\rangle^2u_0\in L^1(\mathbb R^2)$ and $|u_0||\ln u_0|\in L^1(\mathbb R^2)$. If $M<8\pi$ then the solution of \eqref{1.1} exists globally in time, see e.g. \cite{Nagai,Blanchet-2006,Dolbeault-2004}. If $M>8\pi$, then the positive solution of \eqref{1.1} blows up in finite time \cite{Kurokiba-2003,Raphael-2014}.
A detailed construction of blow-up solutions $u(t)$ was done in \cite{He-Ve,Ve,Raphael-2014,Collot-Ghoul-2022,Collot-2022} where the pattern of a collapsing stationary state $U$ was obtained:
\begin{equation}\label{CG}
u(t,x)=\frac{1}{\lambda^2(t)}U\bigg(\frac{x}{\lambda(t)}\bigg)+\tilde u(t,x),\qquad U(x)=\frac{8}{(1+|x|^2)^{2}},
\end{equation}
where $\lambda(t)=2e^{-\frac{2+\gamma}{2}}\sqrt{T-t}e^{-\sqrt{\frac{|\log(T-t)}{2}}}(1+o(1))$ and $\lambda^{2}(t)\|\tilde u(t)\|_{L^\infty}\to 0$ as $t\to T$. For radial nonnegative solutions of \eqref{1.1}, Mizoguchi \cite{Mizoguchi-2022} showed that \eqref{CG} is the only blowup mechanism. At the threshold $M=8\pi$,  radial solutions exist globally in time \cite{Biler-2006}. In \cite{Blanchet-2008,Davila-2020,Ghoul-2018}, the existence of infinite time blowup solutions of $8\pi$ mass was demonstrated.

\subsection{The higher dimensional case} \label{subsec:patterns}

In dimensions $d\geq 3$ the system \eqref{1.1} is mass-supercritical. Indeed, the scaling transformation \eqref{scaling} preserves the $L^{\frac{d}{2}}$-norm, i.e. $||u_\lambda(0)||_{L^{\frac{d}{2}}(\mathbb{R}^d)}=||u(0)||_{L^{\frac{d}{2}}(\mathbb{R}^d)}$, whose Lebesgue exponent is larger than $1$, the exponent associated to the conserved mass. Initial data which are small enough in critical spaces produce global solutions \cite{Biler-2018,Calvez-2012,Corrias,Lemarie}. Large data which are concentrated lead to finite time blowup \cite{Corrias,Calvez-2012,Nagai-1995}. A striking difference with the mass-critical case is that solutions of \eqref{1.1} can blow up with any arbitrary mass since $M(u_\lambda(0))=\lambda^{d-2}M(u(0))$.

Some properties of general blow-up solutions have been obtained, mostly for radial and non-negative solutions of \eqref{1.1}. In dimensions $d\in[3,9],$ any such solution with an initial data that is radially decreasing that blows up is of type I \cite{Mizoguchi-2011}. For $d\geq 3$, any such solution which is of type I is asymptotically backward self-similar, see \cite{Giga-Mizoguchi-Senba-2011}. The profile at blow-up time of blowup solutions with radially nonincreasing initial data satisfy $u(x,T)\sim c |x|^{-2}$, see \cite{Souplet-Winkler}.

Precise examples of radial blow-up solutions with four distinct blow-up patterns are known to exist. First, backward self-similar solutions are those of the form
$$
u(x,t)=\frac{1}{T-t}U\left(\frac{x}{\sqrt{T-t}}\right).
$$
There exists an explicit such solution in all dimensions $d\geq 3$ given by
\begin{equation} \label{id:blow-up-collapsing-sphere}
U_0(x)=\frac{4(d-2)(2d+|x|^2)}{(2(d-2)+|x|^2)^2},
\end{equation}
see \cite{Herrero-1998,Michael}. In dimensions $3\leq d \leq 9$, there exists in fact a countable family of such self-similar solutions
 \begin{equation}\label{gmfs}
u(x,t)=\frac{1}{T-t}U_n\left(\frac{x}{\sqrt{T-t}}\right)
\end{equation}
for $n\geq 0$ an integer \cite{Herrero-1998}. Self-similar solutions of the form \eqref{gmfs} are clearly type I blow-ups. There exist other type I solutions, but with flatter spatial scales
\begin{equation} \label{id:blow-up-flat-typeI}
u(x,t)=\frac{1}{T-t}F\left(\frac{x}{\sqrt{T-t}|\ln T-t|^{\alpha_d}}\right)+\tilde u(x,t),
\end{equation}
with $(T-t)\| \tilde u(t)\|_{L^\infty(\mathbb R^d)} \to 0$ as $t\to T$, for $d=3,4$ with $\alpha_3=1/6$ and $\alpha_4=1/4$, see \cite{Nguyen-Zaag-2023} (where the expression for $F$ can be found). Type I solutions with spatial scales of the form $(T-t)^\beta$ for $\beta>1/2$ are conjectured to exist for $d\geq 3$ in \cite{Nguyen-Zaag-2023}.

There also exist type II blow-up for all $d\geq 3$. If $d\geq 11$ type II solutions concentrating a stationary state $Q$, of the form
\begin{equation} \label{id:blow-up-collapsing-stationary-state}
u(x,t)=\frac{1}{\lambda^2(t)} Q\left(\frac{x}{\lambda(t)}\right)+\tilde u(x,t), \quad \mbox{where } \frac{\lambda(t)}{\sqrt{T-t}}\to 0
\end{equation}
and $ \lambda^2(t)\| \tilde u(t)\|_{L^\infty(\mathbb R^d)} \to 0$ as $t\to T$, are constructed in \cite{Mizo-Senba}. Finally, there exist type II solutions that concentrate their mass near a sphere of radius $R(t)=c(u)(T-t)^{1/d}$ that shrinks to a point, of the form
\begin{equation} \label{id:blow-up-collapsing-sphere}
u(x,t)=\frac{1}{(T-t)^{2\frac{d-1}{d}}}W\left(\frac{|x|-R(t)}{(T-t)^{\frac{d-1}{d}}} \right)+\tilde u(x,t),
\end{equation}
with $W$ a one dimensional traveling wave and $(T-t)^{2\frac{d-1}{d}}\| \tilde u(t)\|_{L^\infty(\mathbb R^d)}\to 0$ as $t\to T$, see \cite{Collot-2023}, and \cite{Herrero-1997} for an earlier formal construction. This pattern also emerges in certain blowup solutions of the nonlinear Schr$\mathrm{\ddot{o}}$dinger equation \cite{Merle-2014,Fibich-2007}. There exists a particular self-similar solution with profile $U_*$ which from numerical simulations lies at the threshold between the collapsing sphere blow-up \eqref{id:blow-up-collapsing-sphere} and blow-up with self-similar profile $U_0$ \cite{Michael}.

We conjecture that there are no other blow-up patterns for localised radial solutions other than these four ones.

\subsection{Main results}

In this article we address the problem of the stability of the blow-up patterns. This understanding of the local phase portrait is a first step towards the global one. The self-similar blow-up solution $U_0$ has been shown to be stable by radial Schwartz perturbations with small $H^3$ norm by Glogi$\mathrm{\acute{c}}$ and Sch$\mathrm{\ddot{o}}$rkhuber \cite{Glogic} using a semigroup approach. Very recently, Li and Zhou \cite{Li-Zhou} constructed a blowup solution of the Keller-Segel-Navier-Stokes system, by expanding further the semi-group approach to the stability analysis of $U_0$. Let $N_0$ denote the number of non-positive eigenvalues of the linearized operator around $U_0$, see Proposition \ref{self profile} for an exact definition. It was showed numerically that $N_0=1$ for $d\ge3$ in \cite{Michael}, and this was proved rigorously for $d=3$ in \cite{Glogic}.
 
 Our first result shows the radial stability of $U_0$ by general perturbations.

\begin{theorem}[Radial stability of $U_0$] \label{theorem-1}
Assume $d=3$, or $d\ge 4$ and $N_0=1$. Then there exist $\epsilon>0$ such that for any radial $u_0=U_0+\tilde u_0$ with $\| \tilde u_0\|_{L^\infty(\mathbb R^d)}\leq \epsilon$, the solution to \eqref{1.1} blows up in finite time $0<T<+\infty$ (depending on $u_0$) and can be decomposed as
$$
u(t,x)=\frac{1}{T-t} U_0\left(\frac{x}{\sqrt{T-t}} \right)+\tilde{u}(t,x)
$$
where:\\
$(\mathrm{i})$ \emph{Behaviour of the maximum norm}: there holds the asymptotic stability of $U_0$,
\begin{equation}\label{norms-1}
\lim_{t\to T}(T-t)||\tilde{u}(t)||_{L^\infty}=0.
\end{equation}
In particular, the blow-up is type I.\\
$(\mathrm{ii})$ \emph{Convergence of the solution:} there exists a function $u^*$ such that $u(t,x)\to u^*(x)$ for any $x\neq 0$ as $t\to T$. For any $p\in [1,\frac 32)$, assuming $u_0\in L^p(\mathbb R^3)$, we have $u^*\in L^p(\mathbb{R}^3)$ and
\begin{equation}\label{convergences-1}
\lim_{t\to T}||u(t)-u^*||_{L^p}=0.
\end{equation}
$(\mathrm{iii})$ \emph{Regularity of the blow-up time}: the blow-up time $T$ is a Lipschitz function of $u_0$ with respect to the $L^\infty$ norm, i.e. there exists $C>0$ such that $|T(u_0)-T(u_0')|\leq C \| u_0-u_0'\|_{L^\infty(\mathbb R^d)}$ for any two such functions $u_0$ and $u_0'$.
\end{theorem}

\begin{remark} \label{remark1}

The result of Theorem \ref{theorem-1} naturally extends to other spaces for which the system \eqref{1.1} is well posed. This is the case, for example, for $L^p(\mathbb R^d)$ for $d/2<p<d$. Indeed, there exists $t_0>0$ such that if $\| \tilde u_0\|_{L^p}$ is small enough then the solution of \eqref{1.1} with data $U_0+\tilde u_0$ satisfies $u(t_0)=U_0+\tilde u(t_0)$ and the map $\tilde u_0\mapsto \tilde u(t_0)$ is Lipschitz from $L^p$ to $L^\infty$ in a neighbourhood of the origin by parabolic regularization. By applying Theorem \ref{theorem-1} to $u(t_0)$, we infer that the analogue of Theorem \ref{theorem-1} holds true in $L^p(\mathbb R^d)$.

\end{remark}

Our second result shows the Lipschitz codimensional stability of the countable family of self-similar profiles \eqref{gmfs}, proving that $U_n$ is the blow-up profile of a class of initial data residing on a manifold of radial initial data of codimension $N_n-1$. The integer $N_n$ denotes the number of radial instable modes $\phi_j$ of the linearized operator around $U_n$ with the convention that $\phi_{N_n}=\Lambda U_n$, see Proposition \ref{self profile}. We introduce their localization $\varphi_j=\phi_j \chi(x/R)$ for $\chi $ a cut-off function and a large enough $R$, and introduce
$$
V_n=\left\{v_0\in L^\infty \mbox{ radial }, \quad \int_{\mathbb R^d} v_0 \varphi_j dx =0 \quad \mbox{for }1\leq j \leq N_n-1\right\}.
$$

\begin{theorem}(Finite Lipschitz codimensional radial stability of $U_n$).\label{theorem}
For all $3\leq d \leq 9$ and $n\ge1$, there exist functions $c_j:V_n\to \mathbb R$ for $1\leq j \leq N_n-1$ that are Lipschitz with respect to the $L^\infty$ topology and with $c_j(0)=0$, such that the following hold true. For any $v_0\in V_n$ with $\| v_0\|_{L^\infty}$ small enough, the initial data
$$
u_0=U_n+v_0+\sum_{j=1}^{N_n-1}c_j(v_0)\varphi_j
$$
produces a solution to \eqref{1.1} that blows up in finite time $0<T<+\infty$ (depending on $v_0$) and can be decomposed as
$$
u(t,x)=\frac{1}{T-t} U_n\left(\frac{x}{\sqrt{T-t}} \right)+\tilde{u}(t,x)
$$
where:\\
$(\mathrm{i})$ \emph{Behaviour of the $L^\infty$ norm}: there holds the asymptotic stability of the self similar profile
\begin{equation}\label{norms}
\lim_{t\to T}(T-t)||\tilde{u}(t)||_{L^\infty}=0.
\end{equation}
In particular, the blow-up is type I.\\
$(\mathrm{ii})$ \emph{Convergence of the solution:} there exists a function $u^*$ such that $u(t,x)\to u^*(x)$ for any $x\neq 0$ as $t\to T$. For any $p\in [1,\frac 32)$, assuming $u_0\in L^p(\mathbb R^3)$, we have $u^*\in L^p(\mathbb{R}^3)$ and
\begin{equation}\label{convergences}
\lim_{t\to T}||u(t)-u^*||_{L^p}=0.
\end{equation}
$(\mathrm{iii})$ \emph{Regularity of the blow-up time}: the blow-up time $T$ is a Lipschitz function with respect to the $L^\infty$ topology, i.e. $\| T(u_0)-T(u_0')\|_{L^\infty}\lesssim \| u_0-u_0'\|_{L^\infty} \lesssim \| v_0-v_0'\|_{L^\infty}$ for two such functions $u$ and $u'$.
\end{theorem}

\begin{remark}
\begin{itemize}
\item There are $N_n$ instable modes, and the codimension of the set of Theorem \ref{theorem} is $N_n-1$. This is because the eigenmode $\phi_{N_n}=\Lambda U_n$ corresponds to an instability by change of the blow-up time, but which is not an instability of the blow-up pattern.
\item The profile $U_1$ has from numerical simulations $N_1=2$ instable modes \cite{Michael}. So if $N_1=2$ holds rigorously the set of Theorem \ref{theorem} is a hypersurface.
\item The profile $U_*$ mentioned in Subsection \ref{subsec:patterns} has from numerical simulations $N_*=2$ instable modes \cite{Michael}. So again, if $N_*=2$ holds rigorously the set of Theorem \ref{theorem} is a hypersurface.
\end{itemize}
\end{remark}

\subsection{Novelties}

Our result extends the stability result of \cite{Glogic} of $U_0$ to all radial self-similar profiles. We obtain the stability of $U_0$ in a larger space of initial data (see Remark \ref{remark1}). In particular, Theorem \ref{theorem-1} and Theorem \ref{theorem} shows that any of these radial self-similar profiles can be the blow-up profile of a solution starting from $C^\infty_c$ initial data. Our methods are different than the spectral ones of \cite{Glogic}, and we rely on weighted energy estimates to control the flow in parabolic variables. We expand the analysis of \cite{Collot-Pierre-2019}, with one novelty of using additional sharp weighted $L^\infty$ estimates to simplify the control of nonlinear terms. As an outcome, this refinement of the analysis of \cite{Collot-Pierre-2019} can address further regularity issues, such as the Lipschitz regularity of the blow-up time and of the stable manifold of Theorem \ref{theorem}, and the regularity of the limit at blow-up time $u^*$.

\subsection*{Organization\ of\ the\ paper.} We prove the result for $d=3$ to ease the notation. Our proof adapts directly to other dimensions $d\geq 4$, as our arguments are valid in all dimensions $d\geq 3$. We prove Theorem \ref{theorem-1} and Theorem \ref{theorem} simultaneously, by studying a perturbation of $U_n$ so that these theorems correspond to the cases $n=0$ and $n\geq 1$ respectively.

In section \ref{section2}, we reformulate \eqref{1.1} as a local semilinear heat equation in $\mathbb{R}^5$ by introducing the reduced mass. In section \ref{section3}, by the results in \cite{Michael,Glogic}, we study the spectral gap for the linearized operator corresponding to \eqref{transform} around $\Phi_n$, Additionally, we derive the asymptotic behaviors of the eigenvectors corresponding to negative eigenvalues for this operator. In section \ref{section4}, we first give the geometrical decomposition of the solutions around self-similar profiles. Then, through a bootstrap analysis, we establish that by selecting suitable initial data, there exists a regime wherein the solutions remain trapped indefinitely. This constitutes the proof of Theorem \ref{theorem}. Lastly, in the Appendix, we present the coercivity estimate utilized in Section \ref{section4}.

\section{Notations}\label{section2}

\subsection{The partial mass framework}

We denote $|x|$ by $r$ and define the so-called reduced mass as follows:
\begin{equation}\label{transfrom}
w(t, r):=\frac{1}{2 r^{3}} \int_{0}^{r} u(t,\tilde{r}) \tilde{r}^{2} d \tilde{r}.
\end{equation}
The system \eqref{1.1} can be rewritten in
the form of a single local semilinear heat equation in $w$:
\begin{equation}\label{main}
\left\{\begin{aligned}
&\partial_t w= \Delta w+\Lambda^{\prime} w^{2}+6 w^{2}, \\
&w(0, \cdot)=w_{0},
\end{aligned} \quad \text { in } \mathbb{R}^5 .\right.
\end{equation}
Here
$$
\Lambda^{\prime} w(x):=x \cdot \nabla w(x),
$$
and $ w_{0}(r)=\frac{1}{2 r^{3}} \int_{0}^{r} {u}_{0}(\tilde{r}) \tilde{r}^{2} d \tilde{r}$. Furthermore, the self-similar solution \eqref{gmfs} transforms into
$$
w_T(t, x):=\frac{1}{T-t} \Phi\left(\frac{x}{\sqrt{T-t}}\right)\ \ \text{where}\ \ \Phi(r)=\frac{1}{2 r^{3}} \int_{0}^{r} U(\tilde{r}) \tilde{r}^{2} d \tilde{r},
$$
and $\Phi(x)$ is  the self-similar profile of \eqref{main} satisfying
\begin{equation}\label{transform}
 \Delta \Phi-\frac{1}{2}\Lambda\Phi+6\Phi^2+\Lambda^{\prime}\Phi^2=0,
\end{equation}
where operator $\Lambda$ defined in \eqref{operator11}.
The following result is due to \cite{Herrero-1998,Michael}
\begin{lemma}\label{self-similar profile}
Let $3\leq d \leq 9$. There exists a countable family of positive smooth radially symmetric self-similar profiles $\{\Phi_n\}_{n\ge0}$ of \eqref{main}, where $\Phi_0(x)=\frac{2}{2+|x|^2}$,
\begin{equation}\label{largex}
\Phi_n(x)\sim\frac{\tilde{c}_n}{|x|^2} \quad \text{as}\ |x|\to +\infty, \qquad \tilde{c}_n\in(0,2].
\end{equation}
\end{lemma}

\subsection{Notations}
\emph{Weighted spaces.}
We define the derivation operator
$$
D^k:= \begin{cases}\Delta^m, & \text { for } k=2 m, \\ \nabla \Delta^m, & \text { for } k=2 m+1.\end{cases}
$$
We define the scalar product
\begin{equation}\label{scact}
(f,g)_{\rho_n}=\int_{\mathbb{R}^5}f(x)g(x)\rho_n(x)dx,\ \ \rho_n(x)=\tilde{\rho}_n(x)e^{-|x|^2/4},
\end{equation}
where
$$
\tilde{\rho}_n(r)=\exp\left(\int_0^r2\tilde{r}\Phi_n(\tilde{r})d\tilde{r}\right)
,\ \tilde{\rho}_0(x)=\Phi_0(x)^{-2},
$$
 and let $L^2_{\rho_n}$
 be the corresponding weighted $L^2$ space. We let $H^k_{\rho_n}$  be the completion of $C^\infty_c(\mathbb{R}^5)$ for the norm
$$
\|u\|_{H_{\rho_n}^k}=\sqrt{\sum_{j=0}^k\left\|D^j u\right\|_{L_{\rho_n}^2}^2}.$$
\emph{Linearized operator.}
The scaling semi-group on functions $u: \mathbb{R}^3\to \mathbb{R}$:
$$
u_\lambda(x):=\lambda^2u(\lambda x)
$$
has for infinitesimal generator the linear operator
\begin{equation}\label{operator11}
\Lambda u:=2u+x\cdot \nabla u=\left.\frac{\partial}{\partial \lambda}(u_\lambda)\right.|_{\lambda=1}.
\end{equation}
We define the linearized operator corresponding to \eqref{transform} around $\Phi_n$ by
\begin{equation}\label{operator}
L_n:=-\Delta+\frac{1}{2} \Lambda-2 \Lambda^{\prime}(\Phi_n \cdot)-12 \Phi_n,
\end{equation}
and we equip $L_n$ with domains $$\mathcal{D}(L_n):=C^{\infty}_{c,rad}(\mathbb{R}^5)\subset H^2_{\rho_n}(\mathbb{R}^5).$$
\emph{General notations.}
Let $\chi(x)$ denote a smooth radially symmetric function with
\begin{equation}\nonumber
\chi(x):= \begin{cases}1 & \text { for } \quad|x| \leq \frac{1}{4} \\ 0 & \text { for } \quad|x| \geq \frac{1}{2}\end{cases}
\end{equation}
and for $A>0$ we define
$$
\chi_A(x)=\chi\left(\frac{x}{A}\right).
$$
Let $B_R$ denote the ball of radius $R > 0$ centered at the origin in $\mathbb{R}^5$.
We use the notation $a \lesssim b$ if $a\le Cb$ for an independent constant $C$. The constant $C(\cdot)$ denotes the positive constant depends on ``$\cdot$'', and it may vary from line to line.

\section{Eigenfunctions and spectral gap in weighted norms}\label{section3}
The purpose of this chapter is to obtain a spectral gap for the linearized operator corresponding to \eqref{transform} around $\Phi_n$:
$$
L_n:=-\Delta+\frac{1}{2} \Lambda-2 \Lambda^{\prime}(\Phi_n \cdot)-12 \Phi_n.
$$
By \eqref{scact} and \eqref{coercivity}, and considering the local compactness of the Sobolev embeddings $H^1_{\rho_n}(\mathbb{R}^5)\hookrightarrow L^2_{\rho_n}(\mathbb{R}^5)$, we know  that $L_n$ is self adjoint with respect to the $L^2_{\rho_n}$ scalar product and possesses a compact resolvent.

We are first concerned about the asymptotic behavior of the eigenfunctions with negative eigenvalues for the operator $L_n$. We decompose $L_n$ as follows: $$L_n=L_n^\infty+\tilde{L}_n$$ where
\begin{equation}\label{swize}
L_n^\infty=-\Delta+\frac{1}{2}\Lambda,\quad \mbox{and}\quad
\tilde{L}_n=-2\Lambda'(\Phi_n\cdot)-12\Phi_n.
\end{equation}
Given $\lambda<0$ and $r_0>0$ sufficiently large, we define $X_{r_0}$ as the space of functions on $(r_0,+\infty)$ such
that the following norm is finite
\begin{equation}\label{swiland}
||w||_{X_{r_0}}=\sup_{r\ge r_0}\left(r^{2(2-\lambda)}|w|+r^{2(2-\lambda)+1}|\partial_rw|\right).
\end{equation}

\begin{lemma}\label{ODE} The following holds for $\lambda<0$.

\noindent $(\mathrm{i})$ Basis of fundamental solutions:
the equation
$$
(L_n^\infty-\lambda)(u)=0,
$$
has two fundamental solutions with the following asymptotic behavior as $r\to \infty$:
\begin{equation}\label{sap}
u_1(r)= r^{2(\lambda-1)}(1+O(r^{-2})) \quad \mbox{and} \quad u_2(r)=r^{-2\left(\frac{3}{2}+\lambda\right)}e^{\frac{r^2}{4}}(1+O(r^{-2})).
\end{equation}
$(\mathrm{ii})$ Continuity of the resolvent: The inverse
$$
\tau{(f)}=\left(\int_r^{+\infty}fu_2{r'}^4e^{-\frac{{r'}^2}{4}}dr'\right)u_1-\left(\int_r^{+\infty}fu_1{r'}^4e^{-\frac{{r'}^2}{4}}dr'\right)u_2
$$
satisfies
$$
(L_n^\infty-\lambda)(\tau{(f)})=f
$$
and
\begin{equation}\label{motors}
||\tau{(f)}||_{X_{r_0}}\lesssim \sup_{r\ge r_0}r^{2(2-\lambda)}|f|.
\end{equation}
\end{lemma}
\begin{proof}
\emph{$\mathbf{Step\ 1}.$}\  Basis of homogeneous solutions.
 We define the change of variable:
\begin{equation}\label{ysnn}
u(r)=\frac{1}{y^{\frac{3}{2}}}\phi(y),\qquad y=r^2.
\end{equation}
We compute that
\begin{equation}\nonumber
\begin{aligned}
L_n^\infty(u)=-\frac{4}{y^{\frac{3}{2}}}\left\{y\phi''(y)-\left(\frac{1}{2}+\frac{y}{4}\right)\phi'(y)+\frac{1}{8}\phi(y)\right\}
.
\end{aligned}
\end{equation}
We next set
$$
\phi(y)=\nu(z),\qquad z=\frac{y}{4}.
$$
Then we compute that
$$
L_n^\infty(u)=-\frac{1}{y^{\frac{3}{2}}}\left\{z\nu''(z)-\left(\frac{1}{2}+z\right)\nu'(z)+\frac{1}{2}\nu(z)\right\}
.
$$
Therefore, $L_n^\infty(u)=\lambda u$ if and only if
\begin{equation}\label{fervex}
z\frac{d^2\nu}{dz^2}+(b-z)\frac{d\nu}{dz}-a\nu=0,
\end{equation}
where $b=-1/2$ and $a=-1/2-\lambda$. The equation \eqref{fervex} is known as Kummer's equation and has been well studied (see, e.g. \cite{Olver-Lozier},). If the parameter $a$ is not a negative integer, which holds in particular for our choice of $a$ because $\lambda<0$, then the basis of solutions to Kummer's equation consists of the Kummer function $M(a,b,z)$ and the Tricomi function $U(a,b,z)$. Therefore, $\phi(y)$ is a linear combination of the special functions  $M(a,b,z)$ and $U(a,b,z)$, whose asymptotic at infinity are given by:
$$
M(a,b,z)=\frac{\Gamma(b)}{\Gamma(a)}z^{a-b}e^z(1+O(z^{-1})),\ \ U(a,b,z)= z^{-a}(1+O(z^{-1}))\ \ \text{as}\ z\to+\infty.
$$
Then  we conclude the proof of \eqref{sap} by \eqref{ysnn}.

\medskip

\noindent \emph{$\mathbf{Step\ 2}.$}\  Estimate on the resolvent. The Wronskian
$$
W:=u_1'u_2-u_2'u_1
$$
satisfies
$$
W'=\left(\frac{r}{2}-\frac{4}{r}\right)W,\ \ W=\frac{C}{r^4}e^{\frac{r^2}{4}},
$$
where we may without loss of generality assume $C=1$. Next we solve
$$
(L_n^\infty-\lambda)(w)=f.
$$
By the variation of constants, we obtain
$$
w=\left(a_1+\int_r^{+\infty}fu_2{(r')}^4e^{-\frac{{r'}^2}{4}}dr'\right)u_1+\left(a_2-\int_r^{+\infty}fu_1{(r')}^4e^{-\frac{{r'}^2}{4}}dr'\right)u_2.
$$
Then $\tau(f)$ satisfies
$$
(L_n^\infty-\lambda)(\tau(f))=f
$$
with corresponding the choice of $a_1=a_2=0$.

We next estimate the asymptotic behavior of $\tau(f)$. For $r\ge r_0$ large enough, we have
\begin{equation}\nonumber
\begin{aligned}
&r^{2(2-\lambda)}|\tau(f)|\\
&=r^{2(2-\lambda)}\left|\left(\int_r^{+\infty}fu_2{r'}^4e^{-\frac{{r'}^2}{4}}dr'\right)u_1-\left(\int_r^{+\infty}fu_1{r'}^4e^{-\frac{{r'}^2}{4}}dr'\right)u_2\right|\\
&\lesssim r^2\left(\int_r^{+\infty}|f|{r'}^{1-2\lambda}dr'\right)+r^{1-4\lambda}e^{\frac{r^2}{4}}\left(\int_r^{+\infty}|f|{r'}^{2(\lambda+1)}e^{-\frac{{r'}^2}{4}}dr'\right)\\
&\lesssim \sup_{r\ge r_0}r^{2(2-\lambda)}|f|\left\{\left(\int_r^{+\infty}\frac{dr'}{r'^3}\right)r^2+
r^{1-4\lambda}e^{\frac{r^2}{4}}\left(\int_r^{+\infty}{r'}^{2(2\lambda-1)}e^{-\frac{{r'}^2}{4}}dr'\right)\right\}\\
&\lesssim \sup_{r\ge r_0}r^{2(2-\lambda)}|f|,
\end{aligned}
\end{equation}
and
\begin{equation}\nonumber
\begin{aligned}
&r^{2(2-\lambda)+1}|\partial_r\tau(f)|\\
&=r^{2(2-\lambda)+1}\left|\left(\int_r^{+\infty}fu_2{r'}^4e^{-\frac{{r'}^2}{4}}dr'\right)\partial_ru_1-\left(\int_r^{+\infty}fu_1{r'}^4e^{-\frac{{r'}^2}{4}}dr'\right)\partial_ru_2\right|\\
&\lesssim r^2\left(\int_r^{+\infty}|f|{r'}^{1-2\lambda}dr'\right)+r^{1-4\lambda}e^{\frac{r^2}{4}}\left(\int_r^{+\infty}|f|{r'}^{2(\lambda+1)}e^{-\frac{{r'}^2}{4}}dr'\right)\\
&\lesssim \sup_{r\ge r_0}r^{2(2-\lambda)}|f|\left\{\left(\int_r^{+\infty}\frac{dr'}{r'^3}\right)r^2+
r^{1-4\lambda}e^{\frac{r^2}{4}}\left(\int_r^{+\infty}{r'}^{2(2\lambda-1)}e^{-\frac{{r'}^2}{4}}dr'\right)\right\}\\
&\lesssim \sup_{r\ge r_0}r^{2(2-\lambda)}|f|,
\end{aligned}
\end{equation}
then \eqref{motors} is proved.
\end{proof}

We next give the asymptotic behavior of the eigenfunctions with negative eigenvalues of the operator $L_n$.
\begin{lemma}\label{taxi}
Let $\lambda<0$. The solutions to
\begin{equation}\label{viesd}
(L_n-\lambda)(u)=0, \ \ u\in H^1_{\rho_n}(1,+\infty),
\end{equation}
behave for $r\to +\infty$ as
$$
|\partial_ku(r)|\lesssim r^{{2(\lambda-1)}-k},\ \ k=0,1.
$$
\end{lemma}
\begin{proof}
Formally, a  non-zero contribution of $u_2$ (see \eqref{sap}) to $u$ would yield for the following asymptotic
$$
u(r)\sim r^{-3-2\lambda}e^{\frac{r^2}{4}}\ \ \text{as}\ r\to\infty,
$$
which contradicts $u\in H^1_{\rho_n}(1,+\infty)$. Hence, we consider $u$ is a perturbation of ${u}_1$ in the following.
\\
\emph{$\mathbf{Step\ 1}.$}\  Setting up the Banach fixed point. Let $\tilde{u}_1$ such that
$$u=u_1+\tilde{u}_1$$
solves \eqref{viesd}, hence
\begin{equation}\label{viesd3}
(L^\infty_n-\lambda)\tilde{u}_1=-\tilde{L}_n(u_1+\tilde{u}_1).
\end{equation}
We define the map:
\begin{equation}\label{viesd11}
G(\tilde{u}_1)=\tau(-\tilde{L}_n(u_1+\tilde{u}_1)).
\end{equation}
We claim the following bounds: for $r_0>0$ large enough, assume that
$$
\tilde{u}_1\in X_{r_0},
$$
then
\begin{equation}\label{viesd1}
G(\tilde{u}_1)\in X_{r_0},
\end{equation}
and for $\tilde{u}_1, \tilde{u}_2\in X_{r_0},$
\begin{equation}\label{viesd2}
||G(\tilde{u}_1)-G(\tilde{u}_2)||_{X_{r_0}}<\delta||\tilde{u}_1-\tilde{u}_2||_{X_{r_0}}
\end{equation}
for some $0<\delta<1$.

Assume that \eqref{viesd1} and \eqref{viesd2} hold, by the Banach fixed point theorem, there exists a unique solution $\tilde{u}_1$ of \eqref{viesd3} such that
$u=u_1+\tilde{u}_1$
solves \eqref{viesd}. Additionally, $u\in H^1_{\rho_n}(1,+\infty)$.
\\
\emph{$\mathbf{Step\ 2}.$}\  Proof of \eqref{viesd1} and \eqref{viesd2}. Assume
$
\tilde{u}_1\in X_{r_0}.
$
By \eqref{viesd11} we have
$$
(L^\infty_n-\lambda)(G(\tilde{u}_1))=-\tilde{L}_n(u_1+\tilde{u}_1).
$$
By \eqref{motors}, there holds
$$
||G(\tilde{u}_1)||_{X_{r_0}}\lesssim \sup_{r\ge r_0}r^{2(2-\lambda)}|\tilde{L}_n(u_1+\tilde{u}_1)|.
$$
Since $\tilde{u}_1\in X_{r_0}$, combining \eqref{largex}, \eqref{swize} and \eqref{sap}, we obtain
$$
||G(\tilde{u}_1)||_{X_{r_0}}\lesssim 1,
$$
which concludes the proof of \eqref{viesd1}.

We next prove \eqref{viesd2}. For $\tilde{u}_1, \tilde{u}_2\in X_{r_0},$ by \eqref{largex} and \eqref{motors}, if we take $r_0$ large enough, then
\begin{equation}\nonumber
\begin{aligned}
&||G(\tilde{u}_1)-G(\tilde{u}_2)||_{X_{r_0}}\\
&\lesssim \sup_{r\ge r_0}r^{2(2-\lambda)}|\tilde{L}_n(\tilde{u}_1-\tilde{u}_2)|\\
&\lesssim  \sup_{r\ge r_0}r^{2(2-\lambda)}\{|r\partial_r\Phi_n(\tilde{u}_1-\tilde{u}_2)|+|r\Phi_n\partial_r(\tilde{u}_1-\tilde{u}_2)|+|\Phi_n(\tilde{u}_1-\tilde{u}_2)|\}\\
&\lesssim  \sup_{r\ge r_0}\big(r^{2(2-\lambda)}|\tilde{u}_1-\tilde{u}_2|+r^{2(2-\lambda)+1}|\partial_r\tilde{u}_1-\partial_r\tilde{u}_2|\big)\frac{1}{r^2}\\
&\le\delta||\tilde{u}_1-\tilde{u}_2||_{X_{r_0}},
\end{aligned}
\end{equation}
where $0<\delta<1$,
which completes the proof of \eqref{viesd2}.
\\
\emph{$\mathbf{Step\ 3}.$}\  Uniqueness of $H^1_{\rho_n}$ solution. Let $\bar{u}$ be another solution with $$\bar{u}\not\in \text{Span}\{u\}.$$
Assume by contradiction that $\bar{u}\in H^1_{\rho_n}$, then there holds $ W\in L^2_{\rho_n}$ where
$$
W=u\bar{u}'-u'\bar{u}.
$$
By \eqref{largex},
$$
W'=\left(\frac{r}{2}-\frac{4}{r}-2r\Phi_n\right)W,\ \ W=\frac{C}{r^{4+2c_n}}e^{\frac{r^2}{4}}\ \ \text{as}\ r\to\infty,
$$
 we thus know that $W\not\in L^2_{\rho_n},$ which yields a contradiction. The uniqueness of $H^1_{\rho_n}$ solution is proved.
\\
\emph{$\mathbf{Step\ 4}.$}\ Asymptotic behaviour. We know from Step 1 that $\tilde{u}_1\in X_{r_0}$ and $u=u_1+\tilde{u}_1$ is the unique $H^1_{\rho_n}$ solution of \eqref{viesd}.  By \eqref{swiland} and \eqref{sap}, for $k=0,1$, we obtain
$$
|\partial_ku(r)|\lesssim r^{{2(\lambda-1)}-k},
$$
which concludes the proof.
\end{proof}

Observing that the symmetry group of dilations and translations generates the explicit eigenfunction
$$
L_n\Lambda\Phi_n=-\Lambda\Phi_n,
$$
and combining Lemma \ref{taxi} with the results from \cite{Michael,Glogic} (See Theorem 5 of \cite{Michael} and Proposition 4.3 of \cite{Glogic}), we have the following proposition.

\begin{proposition}\label{self profile}
For each $n\ge0$, the operator $L_n$ has $N_n$ nonpositive eigenvalues with\\
$$
\begin{aligned}
& N_0(d=3)=1 \\
& N_0(d)=1 \text { for } d \geq 4 \text { from numerical simulations (non mathematically rigorous), } \\
& N_1(d)=2 \text { for } d \geq 3 \text { from numerical simulations (non mathematically rigorous). }
\end{aligned}
$$
\\
$(\mathrm{i})$ Eigenvalues. The spectrum of $L_n$ is given by distinct eigenvalues $\{- \mu_{j,n}\}_{1\leq j \leq N_n}\cup \{ \lambda_{0,n},\lambda_{1,n},...\}$ with
$$
\mu_{j,n}\geq 0\mbox{ for }1\leq j\leq N_n, \quad -\mu_{1,n}=-1 \quad\mbox{and}\quad 0<\lambda_{0,n}<\lambda_{1,n}<\cdots < \lambda_{j,n}\underset{j\to \infty}{\to}+\infty.
$$
$(\mathrm{ii})$ Eigenvectors. The eigenvectors corresponding to the eigenvalues $(-\mu_{j,n})_{1\le j\le N_n}$  are spherically symmetric bounded functions denoted as $\psi_{j,n}$, with
$$
||\psi_{j,n}||_{L^2_{\rho_n}}=1,\ \psi_{1,n}=\frac{\Lambda \Phi_n}{||\Lambda \Phi_n||_{L^2_{\rho_n}}}.
$$
Moreover, by Lemma \ref{taxi}, there holds as $r\to \infty$
\begin{equation}\label{taxiss}
|\partial_k\psi_{j,n}(r)|\lesssim r^{{-2(\mu_{j,n}+1)}-k},\ \ k=0,1.
\end{equation}
$(\mathrm{iii})$ Spectral gap. There exist some constants $c_n> 0$ such
that for all $u\in H^1_{\rho_n}(\mathbb{R}^5)$,
\begin{equation}\label{spectral gap}
(L_nu,u)_{\rho_n}\ge c_n||u||^2_{H^1_{\rho_n}}-\frac{1}{c_n}\sum_{j=1}^{N_n}(u,\psi_{j,n})^2_{\rho_n}.
\end{equation}
\end{proposition}

\begin{remark}
We conjecture that there is no zero eigenvalue for all $n\geq 0$. For $n=0$ only $-1$ belongs to the nonpositive spectrum, and for $n=1$ only $-1$ and another negative eigenvalue, see \cite{Michael}. For large $n$, we believe that there is no eigenvalue in $(-1,0]$, which could be proved following the ideas of \cite{Collot-Pierre-2019}.
\end{remark}

\section{Dynamical control of the flow}\label{section4}

From now on, $n$ is fixed ($n\ge0$). For simplicity, we omit the $n$ subscript and write $\psi_j$, $\mu_j$, $\rho$, $N$ and $\lambda_j$ instead.

\subsection{Renormalisation}

We define the $L^\infty$ tube around the renormalized versions of $\Phi_n$:
$$
X_\delta=\bigg\{w=\frac{1}{\mu^2}(\Phi_n+v)\bigg(\frac{x}{\mu}\bigg),\  \mu>0,\ ||v||_{L^\infty}<\delta\bigg\}.
$$
\begin{lemma}\label{Implicit}(Geometrical decomposition).
 There exist  $\delta>0$  and $ C>0$  such that any  $w \in X_{\delta}$  has a unique decomposition
$$w=\frac{1}{\lambda^{2}}\left(\Phi_n+\sum_{j=2}^{N}a_j\psi_j+\varepsilon\right)\bigg(\frac{x}{\lambda}\bigg),$$
where  $\varepsilon$  satisfies the orthogonality condition
$$ \left(\varepsilon,\psi_j\right)_{\rho}=0,\ \ 1\le j\le N,$$
the parameters $\lambda$ and $a_j$ being smooth on  $X_{\delta}$, and with
$$\|\varepsilon\|_{L^{\infty}}+\sum_{j=2}^{N}|a_j| \leq C .$$
\end{lemma}

\begin{proof}
It is a standard and classical consequence of the implicit function theorem. We refer to the proof of Lemma 4.2. in \cite{Collot-Pierre-2019} whose arguments apply directly to the present case.
\end{proof}

\subsection{Description of the initial datum}. We define the functions
$$
\theta_j(r)=\frac 12 \int_r^\infty \psi_j (\tilde r)\rho(\tilde r)\tilde r d\tilde r
$$
for $1\leq j \leq N$.  They are Schwartz functions by \eqref{taxiss} (and since a similar algebraic decay can be propagated to higher order derivatives). If $\bar \varepsilon_0$ is the partial mass of a function $\bar u_0$, i.e. $\bar \varepsilon_0(r)=(2r^{3})^{-1}\int_0^r \bar u_0(\tilde r)\tilde r^2 d\tilde r$, then we have by Fubini that
$$
(\bar \varepsilon_0,\psi_j)_{\rho} = \int_{\mathbb R^3} \bar u_0 \theta_j dx.
$$
We define the set of initial perturbations whose partial mass is orthogonal to the unstable eigenmodes (except the first one generated by time translation invariance) as
\begin{align}
\label{id:def-En} E_n &= \Big\{ \bar u_0\in L^\infty(\mathbb R^3), \quad \int_{\mathbb R^3} \bar u_0 \theta_j dx=0 \mbox{ for }2\leq j\leq  N \\
\nonumber & \qquad  \Leftrightarrow (\bar \varepsilon_0,\psi_j)_{\rho}=0  \mbox{ for }2\leq j\leq  N \mbox{ for }\bar \varepsilon_0(r)=(2r^{3})^{-1}\int_0^r \bar u_0(\tilde r)\tilde r^2 d\tilde r \Big\}.
\end{align}
By convention, $E_0=L^\infty$ for $n=0$. In order to be able to produce compactly supported initial data, we pick $R>0$ and define for $1\leq j \leq N$ the localized partial mass of the unstable modes as
\begin{equation} \label{id:def-bar-psij}
\bar \psi_j(r)=\psi_j (r)\chi_R(r)-\frac{1}{r^3}\int_0^r \tilde r^3 \psi_j(\tilde r)\chi_R'(\tilde r)d\tilde r
\end{equation}
and the corresponding unstable modes in original variables
\begin{equation} \label{id:def-bar-phij}
\bar \phi_j(r) = \frac{2}{r^2}\partial_r (r^3 \psi_j) \chi_R.
\end{equation}
Note that $\bar\phi_j\in C^\infty_c(\mathbb R^3)$.

We will focus on solutions of \eqref{main} that are a suitable perturbation of $\Phi_n$ in the following sense. We first pick $\bar \lambda_0>0$ small enough. We choose $\bar u_0\in E_n$ small enough. Then we consider for $(\bar a_{j,0})_{2\leq j \leq N}$ small the initial data
\begin{equation} \label{id:initial-data-original-variables}
u_0(x)= \frac{1}{\bar \lambda_0^2}\left(U_n+\tilde u_0\right)\left(\frac{x}{\bar \lambda_0} \right)
\end{equation}
where
\begin{equation} \label{id:initial-data-original-variables-2}
\tilde u_0(y)=\bar u_0(y)+\sum_{j=2}^N \bar a_{j,0} \bar \phi_j(y).
\end{equation}
By convention, the sum is empty for $n=0$ in which case $\tilde u_0(y)=\bar u_0(y)$. We denote by $w_0(r)=(2r^{3})^{-1}\int_0^r u_0(\tilde r)\tilde r^2 d\tilde r$ its associated partial mass. By applying Lemma \ref{Implicit}, there exist $\lambda_0>0$ and $(a_{j,0})_{2\leq j \leq N}$ such that it can be rewritten under the form
\begin{equation}\label{1.2}
w_0=\frac{1}{\lambda_0^2}\left(\Phi_n+v_0\right)\left(\frac{x}{\lambda_0}\right),
\end{equation}
where
\begin{equation}\label{1.2-2}
v_0=\sum_{j=2}^{N}a_{j,0}\psi_j+\varepsilon_0
\end{equation} and
$\varepsilon_0$ is a radial function satisfying
\begin{equation}\label{orth}
\left(\varepsilon_0,\psi_j\right)_{\rho}=0,\ \ 1\le j\le N.
\end{equation}
Some properties of the above change of decomposition are stated in the following lemma.

\begin{lemma}[Change of decomposition] \label{lem:change-decomposition}
Fix $\bar \lambda_0>0$. For $R>0$, consider the map which to $(\bar u_0,(\bar a_{j,0})_{2\leq j\leq N})$ (whose corresponding initial data is \eqref{id:initial-data-original-variables}) associates $( \varepsilon_0,( a_{j,0})_{2\leq j\leq N},\lambda_0)$ (given by the decomposition \eqref{1.2}-\eqref{orth} of the partial mass of this initial data). Then this map is well defined for $\| \bar u_0\|_{L^\infty}$ and $|(\bar a_{j,0})_{2\leq j\leq N}|$ small enough. It is continuous from $E_n \times \mathbb R^{N-1}$ into $\{\varepsilon_0 \in L^\infty, \ y\cdot\nabla  \varepsilon_0 \in L^\infty\}\times \mathbb R^{N}$ equipped with its natural topology. Its second component $(\bar u_0,(\bar a_{j,0})_{2\leq j\leq N})\mapsto (( a_{j,0})_{2\leq j\leq N},\lambda_0)$ is $C^\infty$ from $E_n \times \mathbb R^{N-1}$ into $\mathbb R^n$.
Moreover, the jacobian of the second and third component of this map at $(0,0)$ satisfies
\begin{align}
& \label{id:change-decomposition-jacobian} (\partial_{\bar a_j}(a_{i,0}))_{2\leq i,j\leq N}(0,0)=\textup{Id}_{N-1}+o_{R\to \infty}(1),\\
& \label{id:change-decomposition-jacobian-3} \partial_{\bar a_j}(\lambda_{0})(0,0)=0 \qquad \mbox{for }2\leq j \leq N,\\
& \label{id:change-decomposition-jacobian-2} \partial_{\bar u_0}(a_{i,0})(0,0)=0  \qquad \mbox{for }2\leq i \leq N .
\end{align}
\end{lemma}

\begin{proof}[Proof of Lemma \ref{lem:change-decomposition}]

\noindent \textbf{Step 1}. \emph{Regularity properties}. The fact that this map is well defined for $\| \bar u_0\|_{L^\infty}$ and $|(\bar a_{j,0})_{2\leq j\leq N}|$ small enough for follows directly from Lemma \ref{Implicit}. Moreover, Lemma \ref{Implicit} also gives that the map $(\bar u_0,(\bar a_{j,0})_{2\leq j\leq N})\mapsto ( (\bar a_{j,0})_{2\leq j\leq N},\lambda_0)$ is $C^\infty$. We have by combining \eqref{id:initial-data-original-variables} and \eqref{1.2}
\begin{equation} \label{id:change-decomposition-jacobian-tech}
\varepsilon_0(y)= \frac{\lambda_0^2}{\bar \lambda_0^2}\left(\Phi_n+\bar w_0+\sum_{2}^N \bar a_{j,0}\bar \psi_j\right)\left(\frac{\lambda_0}{\bar \lambda_0}y\right)-\Phi_n(y)-\sum_{2}^N a_{j,0} \psi_j(y).
\end{equation}
where $\bar w_0(r)=(2r^{3})^{-1}\int_0^r \bar u_0(\tilde r)\tilde r^2 d\tilde r$. The above formula, the continuity of the maps $( a_{j,0})_{2\leq j\leq N}$ and $\lambda_0$ imply that the map $(\bar u_0,(\bar a_{j,0})_{2\leq j\leq N})\mapsto \varepsilon_0$ is continuous into $\{\varepsilon_0 \in L^\infty, \ y.\nabla  \varepsilon_0 \in L^\infty\}$.

\noindent \textbf{Step 2}. \emph{Jacobian at $(0,0)$}. Note that for $\bar u_0=0$ and $(\bar a_{j,0})_{2\leq j \leq N}=(0,...0)$ we have $\lambda_0=\bar \lambda_0$ and $( a_{j,0})_{2\leq j \leq N}=(0,...0)$. Hence, we have for $\| \bar u_0\|_{L^\infty}$ and $(\bar a_{j,0})_{2\leq j \leq N}$ small that $\frac{\lambda_0}{\bar \lambda_0}-1=O(\| \bar u_0\|_{L^\infty}+|(\bar a_{j,0})_{2\leq j \leq N}|)$ and $( a_{j,0})_{2\leq j \leq N}=O(\| \bar u_0\|_{L^\infty}+|(\bar a_{j,0})_{2\leq j \leq N}|)$ as these maps are differentiable from Step 1. Performing a Taylor expansion in \eqref{id:change-decomposition-jacobian-tech} one shows
\begin{align*}
\varepsilon_0& =  \left(\frac{\lambda_0}{\bar \lambda_0}-1\right)\Lambda \Phi_n+\frac{\lambda_0^2}{\bar \lambda_0^2}\bar w_0\left(\frac{\lambda_0}{\bar \lambda_0}y\right)+\sum_{2}^N \bar a_{j,0}\bar \psi_j-\sum_{2}^N a_{j,0} \psi_j\\
&\qquad  +O_{L^\infty}(\|\bar u_0\|_{L^\infty}^2+|(\bar a_{j,0})_{2\leq j \leq N}|^2).
\end{align*}
Taking the $L^2_{\rho}$ scalar product of the above identity with $\psi_i$ for $2\leq i \leq N$, using that $(\psi_i,\psi_j)_{\rho}=\delta_{i,j}$ and $( \Lambda \Phi_n,\psi_i)_{\rho}=0$ from Proposition \ref{self profile}, as well as $(\bar w_0,\psi_i)_{\rho}=0$ as $\bar u_0\in E_n$ and $\frac{\lambda_0}{\bar \lambda_0}-1=O(\| \bar u_0\|_{L^\infty}+|(\bar a_{j,0})_{2\leq j \leq N}|)$, one has
$$
a_{i,0}=\bar a_{i,0}+\sum_2^N \bar a_{j,0}((\bar \psi_j-\psi_j,\psi_i)_{\rho}+O_{L^\infty}(\|\bar u_0\|_{L^\infty}^2+|(\bar a_{j,0})_{2\leq j \leq N}|^2).
$$
The above identity implies \eqref{id:change-decomposition-jacobian-2}. As $(\bar \psi_j-\psi_j,\psi_i)_{\rho}\to 0$ as $R\to \infty$, it also implies \eqref{id:change-decomposition-jacobian}. The proof of \eqref{id:change-decomposition-jacobian-3} is similar and we omit it.
\end{proof}

We now fix $\bar u_0\in E_n$ small enough, and want to show the existence of initial values for the parameters $(\bar a_{j,0})_{2\leq j \leq N}$ such that the initial data \eqref{id:initial-data-original-variables} produces a solution that converges to $U_n$ after suitable rescaling. For that, we will work directly in the partial mass setting, with the initial data \eqref{1.2}. Indeed, by Lemma \ref{lem:change-decomposition}, \eqref{id:change-decomposition-jacobian} and \eqref{id:change-decomposition-jacobian-2}, the map $(\bar a_{j,0})_{2\leq j \leq N}\mapsto ( a_{j,0})_{2\leq j \leq N}$ is a $C^\infty$ diffeomorphism onto its image that is close to the identity. Hence, it will suffice to show the existence of $( a_{j,0})_{2\leq j \leq N}$ such that the solution $w$ in partial mass converges to $\Phi_n$ after rescaling. Our initial data is chosen as follows. We take
$$
\bar \lambda_0=e^{-\frac{s_0}{2}}
$$
for $s_0$ to be chosen large later on.

\begin{lemma}[Choice of the initial data] \label{lem:initial-data}
For all $K_0,\mu>0$, the following hold true for $s_0$, $R$ and $L$ large enough. Fix $\bar u_0\in E_n$ such that its associated partial mass $\bar w_0$ satisfies
\begin{equation}\label{smallness-original-variables}
\left\| \bar w_0\right\|_{L_\rho^2}+\left\| \bar w_0\right\|_{L^\infty}+\left\|\frac{y^2\bar w_0}{1+e^{-s_0}y^2}\right\|_{L^\infty}+||y\cdot\nabla \bar w_0||_{L^\infty} \leq \frac{K_0}{L} e^{-\mu s_0}.
\end{equation}
Consider $(a_{j,0})_{2\leq j \leq N}$ such that
\begin{equation}\label{smallnessjsue}
\sum_{j=2}^{N}|a_{j,0}|^2\le e^{-2\mu s_0}.
\end{equation}
Let then $\lambda_0$, $u_0$, $v_0$, $\varepsilon_0$ be given by \eqref{id:initial-data-original-variables} and \eqref{1.2-2} via Lemma \ref{lem:change-decomposition}. We have
\begin{equation}\label{smallness0}
\frac 12 e^{-\frac{s_0}{2}}< \lambda_0 < 2 e^{-\frac{s_0}{2}}
\end{equation}
and
\begin{equation}\label{1.5}
\left\|\varepsilon_0\right\|_{L_\rho^2}+\left\|\varepsilon_0\right\|_{L^\infty}+\left\|\frac{y^2v_0}{1+e^{-s_0}y^2}\right\|_{L^\infty}+||y\cdot\nabla v_0||_{L^\infty} \leq K_0 e^{-\mu s_0}.
\end{equation}
Moreover, the map $(a_{j,0})_{2\leq j \leq N}\mapsto u_0$ is continuous from $B(0,e^{-\mu s_0})$ into $L^\infty(\mathbb R^3)$.
\end{lemma}

\begin{proof}

We have by \eqref{id:change-decomposition-jacobian}, \eqref{id:change-decomposition-jacobian-3} and \eqref{id:change-decomposition-jacobian-2} that
\begin{align}
\label{almost-there} & \bar a_{0,j} = a_{j,0} +o_{R\to \infty}(|(a_{j,0})_{2\leq j \leq N}|)+O( (|(a_{j,0})_{2\leq j \leq N}|)^2+\| \bar u_{0}\|_{L^\infty}^2) \quad \mbox{for }2\leq j \leq N,\\
\label{almost-there-2} & \frac{\lambda_0}{\bar\lambda_0}=1+O(|(a_{j,0})_{2\leq j \leq N}|)^2+\| \bar u_{0}\|_{L^\infty}).
\end{align}
Since $\bar u_0=6\bar w_0+2r\partial_r \bar w_0$, we have $\| \bar u_0\|_{L^\infty}\lesssim \frac{K_0}{L}e^{-\mu s_0}$ by \eqref{smallness-original-variables}. Injecting this inequality and \eqref{smallnessjsue} in \eqref{almost-there} and \eqref{almost-there-2} shows
\begin{align}
\label{almost-there-3} & \bar a_{0,j} = a_{j,0} +o_{R\to \infty}(e^{-\mu s_0})+O((1+\frac{K_0}{L})^2e^{-2\mu s_0}) \quad \mbox{for }2\leq j \leq N,\\
\label{almost-there-4} & \frac{\lambda_0}{\bar\lambda_0}=1+O(\frac{K_0}{L}e^{-\mu s_0}+e^{-2\mu s_0}+e^{-4\mu s_0}\frac{K^2_0}{L^2}).
\end{align}
The desired inequalities \eqref{smallness0} and \eqref{1.5} then follows from \eqref{almost-there-4} and from injecting \eqref{almost-there-3} and \eqref{almost-there-4} in \eqref{id:change-decomposition-jacobian-tech}, upon choosing $R$, $L$ and $s_0$ large enough. The continuity statement follows directly from Lemma \ref{lem:change-decomposition}.

\end{proof}

%
%
%
%

\subsection{Bootstrap for the renormalized flow.}
We fix $\bar u_0\in  E_n$ and consider all possible $(a_{j,0})_{2\leq j \leq N}\in B(0,e^{-\mu s_0})$. We assume that the constants $s_0,R,\mu,K_0$ are such that the conclusion of Lemma \ref{lem:initial-data} is valid, and let $u_0$ denote the associated initial data. The constants $s_0\gg1$ and $\mu, K_0>0$ will be adjusted later on. We let $u(t)$ be the corresponding solution to \eqref{1.1} and $w(t)$ denote its partial mass.

As long as the solution $w(t)$ starting from \eqref{1.2} belongs to $X_\delta$, we apply  Lemma \ref{Implicit} to deduce that it can be written
\begin{equation}\label{re}
w(t, x)=\frac{1}{\lambda(t)^2}\left(\Phi_n+v\right)(s, y),\ y=\frac{x}{\lambda(t)},
\end{equation}
where $$ v=\varepsilon+\psi,\ \ \psi=\sum_{j=2}^{N}a_j\psi_j$$
and
$\varepsilon$ is a radial function satisfying
\begin{equation}\label{orthogonality}
\left(\varepsilon, \psi_j\right)_{\rho}=0,\ \ 1\le j\le N,
\end{equation}
and $s$ is the  renormalized time defined by
$$s(t):=\int_0^t\frac{d\tau}{\lambda^2(\tau)}+s_0.$$
As the parameters $\lambda$ and $a_j$ are smooth in  $L^{\infty}$, and as $ w \in C^{1}\left((0, T), L^{\infty}\right)$  from parabolic regularizing effects, the above decomposition is differentiable with respect to time.
Injecting \eqref{re} into \eqref{main} yields the renormalized equation
\begin{equation}\label{renormalized equation}
\partial_{s} \varepsilon+L_n \varepsilon=F+\operatorname{Mod},
\end{equation}
with the modulation term
$$
\text{Mod}=\sum_{j=2}^N[\mu_ja_j-(a_j)_s]\psi_j+\left(\frac{\lambda_s}{\lambda}+\frac{1}{2}\right) \left(\Lambda\Phi_n+\Lambda\psi\right),
$$
and the force terms $F=\tilde{L}(\varepsilon)+NL$ where
\begin{equation}\label{force}
\tilde{L}(\varepsilon)=\left(\frac{\lambda_s}{\lambda}+\frac{1}{2}\right) \Lambda\varepsilon,\ \ NL=6v^2+\Lambda^{\prime}(v^2).
\end{equation}

We claim the following bootstrap proposition.
\begin{proposition}\label{bootstrap}
There exist universal constants $0<\mu,K_0\ll1$, $\delta>0$, $K\gg1$, $K'\gg1$ and $K''\gg1$ such that for all $s_0\ge s_0(K_0,K,K',K'',\mu,\delta)\gg1$ large enough the following holds. Let $\bar{u}_0$ satisfies the assumption of Lemma \ref{lem:initial-data}, and $\varepsilon_0$, $\lambda_0$ be given by Lemma \ref{lem:initial-data}. Then there exists $(a_{2,0},\cdots ,a_{N,0})$ satisfying \eqref{smallnessjsue} such that the solution starting from $w_0$ given by \eqref{1.2}, decomposed according to \eqref{re} for all $s\ge s_0$:\\
- control of the scaling:
\begin{equation}\label{1.7}
0<\lambda(s)<e^{-\mu s};
\end{equation}
- control of the unstable modes:
\begin{equation}\label{chca}
\sum_{j=2}^N\left|a_j\right|^2 \le e^{-2\mu s};
\end{equation}
- control of the exponentially weighted norm:
\begin{equation}\label{1.9}
\left\|\varepsilon\right\|_{L_\rho^2} < K e^{-\mu s};
\end{equation}
- control of the maximum norm:
\begin{equation}\label{2.0}
\left\|\varepsilon\right\|_{L^\infty} < K' e^{-\mu s};
\end{equation}
- control of the weighted maximum norm:
\begin{equation}\label{weihsg}
\left\|\frac{y^2v}{1+e^{-s}y^2}\right\|_{L^\infty}<\delta;
\end{equation}
- control of the maximum norm of $y\cdot\nabla v$:
\begin{equation}\label{ball}
||y\cdot\nabla v||_{L^\infty} < K'' e^{-\mu s}.
\end{equation}
\end{proposition}

The remainder of this paper is dedicated to proving Proposition \ref{bootstrap}, which directly implies Theorem \ref{theorem}.
We next define the exit time
\begin{equation}\label{mmbb}
s^*=\sup\{s\ge s_0: \text{the bounds}\ \eqref{1.7}-\eqref{ball}\ \text{holds on}\ [s_0,s) \},
\end{equation}
and we assume, by contradiction, that
\begin{equation}\label{sdzx}
s^*<+\infty.
\end{equation}
If we choose $K$, $K'$, $K''$, $\delta$ and $s_0$ to be sufficiently large, then $s^*>s_0$. From now on, we study the flow on $[s_0,s^*]$ where $\eqref{1.7}$-$\eqref{ball}$ hold. We employ the bootstrap argument to show that the bounds $\eqref{1.7}$, $\eqref{1.9}$, $\eqref{2.0}$, $\eqref{weihsg}$ and \eqref{ball} always hold. This implies that the unstable modes have grown, and \eqref{chca} holds with the equal sign in the exit time $s^*$.  However, this contradicts the Brouwer fixed-point theorem. Hence, $s^* = +\infty$.
\subsection{Modulation equation}
\begin{lemma}\label{modulation}
We have the following estimate:
\begin{equation}\label{semin}
\bigg|\frac{\lambda_s}{\lambda}+\frac{1}{2}\bigg|+\sum_{j=2}^N|(a_j)_s-\mu_ja_j|\lesssim||\varepsilon||^2_{L^2_\rho}+\sum_{j=2}^N|a_j|^2.
\end{equation}
\end{lemma}
\begin{proof}
\emph{$\mathbf{Step\ 1.}$}\ Law for $a_j$. Take the $L^2_\rho$ scalar product of \eqref{renormalized equation} with $\psi_j$ for $2\le j\le N$, from \eqref{orthogonality} and using the fact that the operator $L_n$  is  self-adjoint for the $L^2_{\rho}$ product, then by
\begin{equation}\label{nonbus}
(\psi_j,\psi_k)_\rho=\delta_{jk}=\begin{cases}1,& \text { for } j=k, \\ 0,& \text { for } j\neq k,\end{cases}
\end{equation}
 we have $$-(\text{Mod},\psi_j)_\rho=(F,\psi_j)_\rho,$$ i.e.,
$$
(a_j)_s-\mu_ja_j=(F,\psi_j)_\rho+\left(\frac{\lambda_s}{\lambda}+\frac{1}{2}\right)(\Lambda \psi,\psi_j)_\rho.
$$
We know from \eqref{chca} that
\begin{equation}\label{chacs}
|(\Lambda \psi,\psi_j)_\rho|\lesssim\sum_{j=2}^N|a_j|\lesssim e^{-\mu s}\le \eta\ll1,
\end{equation}
for $s_0$ large enough.

We next estimate the term $(F,\psi_j)_\rho$. By the definition of $\rho$ in \eqref{scact}, we have
\begin{equation}\label{ldcy}
\nabla \rho(y)=\left(2\Phi_n(y)-\frac{1}{2}\right)y\rho(y),
\end{equation}
then by Cauchy-Schwarz inequality and \eqref{taxiss}, we get
\begin{equation}\label{dells}
\begin{aligned}
|(\tilde{L}(\varepsilon),\psi_j)_\rho|=\left|\left(\left(\frac{\lambda_s}{\lambda}+\frac{1}{2}\right)\Lambda\varepsilon,\psi_j\right)_\rho\right|
\lesssim \left|\frac{\lambda_s}{\lambda}+\frac{1}{2}\right|||\varepsilon||_{L^2_\rho}.
\end{aligned}
\end{equation}
By integrating by parts, Cauchy inequality,  \eqref{ldcy} and \eqref{taxiss}, we obtain
\begin{equation}\label{dellss}
\begin{aligned}
|(NL,\psi_j)_\rho|\le\left|\left(6v^2,\psi_j\right)_\rho\right|+\left|\left(\Lambda^{\prime}(v^2),\psi_j\right)_\rho\right|
\lesssim ||v||_{L^2_\rho}^2\lesssim ||\varepsilon||^2_{L^2_\rho}+\sum_{j=2}^N|a_j|^2.
\end{aligned}
\end{equation}
Combining \eqref{chacs}, \eqref{dells} and \eqref{dellss} yields
\begin{equation}\label{cafss}
|(a_j)_s-\mu_ja_j|\lesssim \left|\frac{\lambda_s}{\lambda}+\frac{1}{2}\right|(\eta+||\varepsilon||_{L^2_\rho})+||\varepsilon||^2_{L^2_\rho}+\sum_{j=2}^N|a_j|^2.
\end{equation}
\emph{$\mathbf{Step\ 2.}$}\ Law for $\lambda$.
Take the $L^2_\rho$ scalar product of \eqref{renormalized equation} with $\psi_1=\frac{\Lambda \Phi_n}{||\Lambda \Phi_n||_{L^2_{\rho}}}$, we get
$$
-\bigg(\frac{\lambda_s}{\lambda}+\frac{1}{2}\bigg)||\Lambda\Phi_n||_{L^2_\rho}=(F,\psi_1)_\rho+\bigg(\frac{\lambda_s}{\lambda}+\frac{1}{2}\bigg)(\Lambda\psi,\psi_1)_\rho.
$$
Then by a same way as in Step 1, we get
\begin{equation}\label{cafss0}
\left|\frac{\lambda_s}{\lambda}+\frac{1}{2}\right|\lesssim \left|\frac{\lambda_s}{\lambda}+\frac{1}{2}\right|(\eta+||\varepsilon||_{L^2_\rho})+||\varepsilon||^2_{L^2_\rho}+\sum_{j=2}^N|a_j|^2.
\end{equation}
Since $0<\eta\ll1$, we prove this current lemma by combing \eqref{1.9}, \eqref{cafss} and \eqref{cafss0}.
\end{proof}

\subsection{Energy estimates with exponential weights}
\begin{lemma}
There holds the bound
\begin{equation}\label{Eys}
\frac{d}{ds}||\varepsilon||_{L^2_\rho}^2+c_n||\varepsilon||_{H^1_\rho}^2\lesssim \sum_{j=2}^N|a_j|^4,
\end{equation}
where the constant $c_n>0$ is given by \eqref{spectral gap}.
\end{lemma}

\begin{proof}
Take the $L^2_\rho$ scalar product of \eqref{renormalized equation} with $\varepsilon$, we have
\begin{equation}\label{xl}
\frac{1}{2}\frac{d}{ds}||\varepsilon||_{L^2_\rho}^2=-(L_n\varepsilon,\varepsilon)_\rho+(F+\text{Mod},\varepsilon)_\rho.
\end{equation}
Combining \eqref{spectral gap} and the orthogonality condition \eqref{orthogonality} yields
\begin{equation}\label{gaps}
-(L_n\varepsilon,\varepsilon)_\rho\le -c_n||\varepsilon||_{H^1_\rho}^2.
\end{equation}
Using Cauchy's inequality with $\delta$:
\begin{equation}\label{Cauchy inequality}
ab\le \delta a^2+C_\delta b^2,\ \ \left(a,b>0,\ \delta>0,\ C_\delta=\frac{1}{4\delta}\right),
\end{equation}
then combining \eqref{chca}, \eqref{semin} and Cauchy-Schwarz inequality, we get
\begin{equation}\label{lyya}
\begin{aligned}
|(\text{Mod},\varepsilon)_\rho|&\le||\varepsilon||_{L^2_\rho}||\text{Mod}||_{L^2_\rho}\lesssim||\varepsilon||_{L^2_\rho}\left(||\varepsilon||^2_{L^2_\rho}+\sum_{j=2}^N|a_j|^2\right)
\\
&\lesssim \delta||\varepsilon||_{L^2_\rho}^2+C_\delta\left(||\varepsilon||^4_{L^2_\rho}+\sum_{j=2}^N|a_j|^4\right),
\end{aligned}
\end{equation}
where $\delta$ is sufficiently small.
By integrating by parts, \eqref{ldcy} and  \eqref{coercivity}, we obtain
\begin{equation}\label{lyysd}
\begin{aligned}
|(\tilde{L}(\varepsilon),\varepsilon)_\rho|&\lesssim\left|\frac{\lambda_s}{\lambda}+\frac{1}{2}\right|\left(\int_{\mathbb{R}^5}\varepsilon^2\rho dy+\left|\int_{\mathbb{R}^5}\varepsilon\nabla\cdot(y\varepsilon\rho) dy\right|\right)\\
&\lesssim \left|\frac{\lambda_s}{\lambda}+\frac{1}{2}\right|||\varepsilon||_{H^1_\rho}^2.
\end{aligned}
\end{equation}
By Cauchy-Schwarz inequality and integrating by parts, we deduce from \eqref{coercivity}, \eqref{taxiss} and \eqref{Cauchy inequality} that
\begin{equation}\label{xl1}
\begin{aligned}
|(NL,\varepsilon)_\rho|&\lesssim
\left|\int_{\mathbb{R}^5}(\varepsilon^2+\psi^2)\varepsilon\rho dy\right|+\left|\int_{\mathbb{R}^5}(\varepsilon^2+\psi^2)\nabla\cdot(y\varepsilon\rho)dy\right|\\
&\lesssim ||\varepsilon||_{L^\infty}||\varepsilon||^2_{H^1_\rho}+\sum^N_{j=2}|a_j|^2||\varepsilon||_{H^1_\rho}
\\
&\lesssim\bigg(||\varepsilon||_{L^\infty}+\delta\bigg)||\varepsilon||^2_{H^1_\rho}+C_\delta\sum^N_{j=2}|a_j|^4.
\end{aligned}
\end{equation}

By amalgamating \eqref{xl}, \eqref{gaps}, \eqref{lyya}, \eqref{lyysd} and \eqref{xl1},   and considering that $s_0$ is sufficiently large while $\delta$ is small enough, we obtain the desired estimate by \eqref{1.9}, \eqref{2.0} and \eqref{semin}.
\end{proof}

\subsection{$L^\infty$ bound of $\varepsilon$}
\begin{lemma}\label{Linfty}
There holds the bound
\begin{equation}\label{dfyusd}
||\varepsilon||_{L^\infty}\le C(K,K_0)e^{-\mu s}.
\end{equation}
\end{lemma}

\begin{proof}
\emph{$\mathbf{Step\ 1}.$}\ $L^\infty$ estimate inside the ball.
We rewrite the equation \eqref{renormalized equation} in the following form:
\begin{equation}\nonumber
\partial_s\varepsilon=\bar{L}\varepsilon+\bar{F},
\end{equation}
where $\bar{L}\varepsilon=\Delta \varepsilon+T\cdot\nabla \varepsilon+V\varepsilon$ and
$$
T=\left[2\Phi_n+\left(\frac{\lambda_s}{\lambda}+\frac{1}{2}\right)+2v-\frac{1}{2}\right]y,$$
$$ V=12\Phi_n+2y\cdot\nabla\Phi_n+2\left(\frac{\lambda_s}{\lambda}+\frac{1}{2}\right)+2y\cdot\nabla\psi-1,\
\bar{F}=6v^2+\Lambda'(\psi^2)+\text{Mod}.
$$
We know from \eqref{1.9}, \eqref{semin}, Lemma \ref{self-similar profile} and Lemma \ref{self profile} that
\begin{equation}\label{fontaines}
||T(y)||_{L^\infty(B_R)}\lesssim1,\ \ ||V(y)||_{L^\infty(B_R)}\lesssim1,
\end{equation}
where $0<R<\infty$ is arbitrary, and
we obtain similarly that
\begin{equation}\label{fontaines0}
||\bar{F}||_{L^\infty(B_R)}\lesssim C(K)e^{-\mu s},
\end{equation}
for $ s_0$ large enough. We know from \eqref{1.9} that
\begin{equation}\label{jjyx}
||\varepsilon||_{L^2(B_R)}\lesssim e^{-\mu s}.
\end{equation}
Combining  \eqref{1.5}, \eqref{fontaines}, \eqref{fontaines0} and \eqref{jjyx},
we employ parabolic regularity \cite{Ladyzhenskaia} to obtain that, there exists a constant $C'=C'(K,K_0)$ such that
\begin{equation}\label{estimate-inside}
||\varepsilon(s)||_{L^\infty(B_R)}\le \frac{1}{2}C'e^{-\mu s}\ \ s\in[s_0,s^*].
\end{equation}
\emph{$\mathbf{Step\ 2}.$}\ $L^\infty$ bound outside the ball.
 We claim that
\begin{equation}\label{estimate-outside}
 ||\varepsilon||_{L^\infty(|y|\ge R)}\le C'e^{-\mu s},\ \ s\in[s_0,s^*].
\end{equation}
We next use parabolic comparison principle to prove the above claim.
Take $\bar{\psi}(y,s)=C'e^{-\mu s}$. Since \eqref{1.5}, \eqref{estimate-inside} and $K_0$ is small enough, on the boundary, there holds
\begin{equation}\label{dfyus}
\bar{\psi}(s_0,y)\ge |\varepsilon(s_0,y)|,\ \ |y|\ge R,
\end{equation}
and
\begin{equation}\label{qdu}
\bar{\psi}(s,R)\ge |\varepsilon(s,R)|,\ \ s\in[s_0,s^*].
\end{equation}
We compute
\begin{equation}\nonumber
\begin{aligned}
\partial_s(\bar{\psi}-\varepsilon)-\bar{L}(\bar{\psi}-\varepsilon)=\bar{\psi}_s-L\bar{\psi}-\bar{F},\ |y|\ge R.
\end{aligned}
\end{equation}
If we take $R$ and $s_0$ large enough, since $\mu>0$ is small enough, then by \eqref{largex}, \eqref{taxiss}, \eqref{chca}, \eqref{1.9} and \eqref{semin},
we have
\begin{equation}\nonumber
\begin{aligned}
\bar{\psi}_s-\bar{L}\bar{\psi}&=
\bigg[1-\mu-12\Phi_n-2y\cdot\nabla\Phi_n-2\left(\frac{\lambda_s}{\lambda}+\frac{1}{2}\right)-2y\cdot\nabla\psi\bigg]\bar{\psi}\ge\frac{\bar{\psi}}{2}.
\end{aligned}
\end{equation}
For $s_0$  large enough, combining \eqref{taxiss}, \eqref{chca}, \eqref{1.9}, \eqref{2.0} and  \eqref{semin} yields
\begin{equation}\nonumber
\begin{aligned}
|\bar{F}|\le6v^2+|\Lambda'(\psi^2)|+\left|\sum_{j=2}^N[\mu_ja_j-(a_j)_s]\psi_j\right|+\left|\frac{\lambda_s}{\lambda}+\frac{1}{2}\right| \left|\Lambda\Phi_n+\Lambda\psi\right|\le \frac{\bar{\psi}}{2}.
\end{aligned}
\end{equation}
 Therefore, if  $R$  and $s_0$ are sufficiently large, we have
\begin{equation}\nonumber
\begin{aligned}
\partial_s(\bar{\psi}-\varepsilon)-\bar{L}(\bar{\psi}-\varepsilon)\ge0\ \ \text{on}\ |y|\ge R.
\end{aligned}
\end{equation}
Then combining \eqref{dfyus} and \eqref{qdu},  we employ parabolic comparison principle to obtain
$$
\varepsilon(s,y)\le \bar{\psi}\ \ \text{on}\ |y|\ge R.
$$
We can show similarly that
$$
-\bar{\psi}\le \varepsilon(s,y)\ \ \text{on}\ |y|\ge R,
$$
which proves the claim.
 We complete the proof by \eqref{estimate-inside} and \eqref{estimate-outside}.
\end{proof}
\subsection{$L^\infty$ bound of $y^2v/(1+e^{-s}y^2)$}
\begin{lemma}\label{egytp}
There holds the bound
\begin{equation}\label{delta}
\left\|\frac{y^2v}{1+e^{-s}y^2}\right\|_{L^\infty}< \frac{\delta}{2}.
\end{equation}
\end{lemma}
\begin{proof}
\emph{$\mathbf{Step\ 1}.$}\ $L^\infty$ bound inside the ball.
We recall
$$
v=\varepsilon+\psi
$$
and we compute the evolution of $v$:
\begin{equation}\label{thies}
\partial_{s} v+L_n v=6v^2+\Lambda'(v^2)+\left(\frac{\lambda_s}{\lambda}+\frac{1}{2}\right) \left(\Lambda\Phi_n+\Lambda v\right).
\end{equation}
Since the eigenvectors $\psi_j$ are bounded for $j=2,\cdots N$, then by \eqref{chca} and \eqref{dfyusd}, we have
\begin{equation}\label{ocrang}
\left\|v\right\|_{L^\infty}\lesssim e^{-\mu s}.
\end{equation}
Take any $0<R<\infty$. We know from \eqref{ocrang} that
\begin{equation}\label{oras2}
\left\|\frac{y^2v(s)}{1+e^{-s}y^2}\right\|_{{L^\infty}(B_R)}\le \frac{\delta}{4},\ \ s\in[s_0,s^*],
\end{equation}
for $s_0$ large enough.\\
\emph{$\mathbf{Step\ 2}.$}\ $L^\infty$ bound outside the ball. We rewrite \eqref{thies} in the following form:
\begin{equation}\label{yjs}
\begin{aligned}
\partial_s{v}+\hat{L}v=\hat{F},
\end{aligned}
\end{equation}
where
\begin{align}
\nonumber & \hat{L}v=-\Delta {v}+\hat{T}\cdot\nabla {v}+\hat{V}{v}, \\
\label{yjs1} & \hat{T}=\left[\frac{1}{2}-2\Phi_n-\left(\frac{\lambda_s}{\lambda}+\frac{1}{2}\right)-2v\right]y,\\
\label{yjs2} & \hat{V}=1-2y\cdot\nabla\Phi_n-12\Phi_n-2\left(\frac{\lambda_s}{\lambda}+\frac{1}{2}\right),\ \ \hat{F}=\left(\frac{\lambda_s}{\lambda}+\frac{1}{2}\right)\Lambda\Phi_n+6v^2.
\end{align}
We take $\nu>0$ and $R(s)=e^{\nu s}$.
Let $$\hat{\psi}(s,y)=\frac{\delta}{2y^2}(1+e^{-s}y^2)(1-\hat{K}e^{-\kappa s}),$$
where $\kappa\ll \mu, \nu$, and $\hat{K}>0$ is a  constant to be fixed later on. By \eqref{1.5} and \eqref{oras2}, on the boundary, we have
\begin{equation}\label{oras3}
|v(s_0,y)|\le \hat{\psi}(s_0,y),\ \ |y|\ge R(s_0),
\end{equation}
and
\begin{equation}\label{qdu1}
|v(s,y)|\le \hat{\psi}(s,y),\ \ |y|= R(s),\ \ s\in[s_0,s^*],
\end{equation}
 for $s_0$ large enough.
By \eqref{largex}, \eqref{chca}, \eqref{1.9}, \eqref{semin} and \eqref{ocrang}, since $\kappa\ll \mu, \nu$, for $y\ge R(s)$,  there holds
 \begin{equation}\label{h2s}
\begin{aligned}
 \partial_s{\hat{\psi}}+\hat{L}{\hat{\psi}}&=
 \frac{\delta}{2y^2}\left(e^{-s}y^2+1\right)\hat{K}\kappa e^{-\kappa s}\\
 &\ \ \ + \frac{\delta}{2y^2}(1- \hat{K}e^{-\kappa s})\left[\frac{2}{y^2}-e^{-s}y^2-2\left(\frac{1}{2}-2\Phi_n-\left(\frac{\lambda_s}{\lambda}+\frac{1}{2}\right)-2v\right)\right]\\
 &\ \ \ +\frac{\delta}{2y^2}(1-  \hat{K}e^{-\kappa s})\left(e^{-s}{y^2}+1\right)\left[1-2y\cdot\nabla\Phi_n-12\Phi_n-2\left(\frac{\lambda_s}{\lambda}+\frac{1}{2}\right)\right]\\
 &= \frac{\delta}{2y^2}\left(e^{-s}y^2+1\right)\hat{K}\kappa e^{-\kappa s}+ \frac{\delta}{2y^2}(1- \hat{K}e^{-\kappa s})\bigg[\frac{2}{y^2}+4\Phi_n+2\left(\frac{\lambda_s}{\lambda}+\frac{1}{2}\right)+4v\\
 &\ \ \ -
 \left(e^{-s}{y^2}+1\right)\left(2y\cdot\nabla\Phi_n+12\Phi_n+2\left(\frac{\lambda_s}{\lambda}+\frac{1}{2}\right)\right)\bigg]\\
 &\ge \frac{\delta}{4y^2}\left(e^{-s}y^2+1\right)\hat{K}\kappa e^{-\kappa s},
\end{aligned}
\end{equation}
 for $s_0$ and $\hat{K}$ large enough. By \eqref{largex}, \eqref{chca}, \eqref{1.9}, \eqref{weihsg}, \eqref{semin} and \eqref{ocrang}, we get
\begin{equation}\label{h2s1}
 |\hat{F}|\lesssim \frac{1}{y^2}e^{-2\mu s}+e^{-\mu s}\left(\frac{1}{y^2}+e^{-s}\right),
\end{equation}
From \eqref{h2s} and \eqref{h2s1}, since $\kappa\ll \mu, \nu$, we have
\begin{equation}\label{h2s01}
 \partial_s({\hat{\psi}}-v)+\hat{L}({\hat{\psi}}-v)=\partial_s{\hat{\psi}}+\hat{L}{\hat{\psi}}-\hat{F}\ge 0,\ \ y\ge R(s),
\end{equation}
for $s_0$ and $\hat{K}$ large enough.
then by \eqref{oras3}, \eqref{qdu1} and parabolic comparison principle yields
\begin{equation}\nonumber
v\le \hat{\psi}< \frac{\delta}{2y^2}(1+e^{-s}y^2),\ \ \text{on}\ |y|\ge R(s).
\end{equation}
Similarly, there holds
\begin{equation}\label{oras3q}
-\frac{\delta}{2y^2}(1+e^{-s}y^2)<-\hat{\psi}\le v ,\ \ \text{on}\ |y|\ge R(s),
\end{equation}
then combining \eqref{oras2} we have
$$
\left\|\frac{y^2v}{1+e^{-s}y^2}\right\|_{L^\infty}< \frac{\delta }{2},
$$
which completes the proof.
\end{proof}

\subsection{$L^\infty$ bound of $y\cdot\nabla v$}
\begin{lemma}\label{acsf}
There holds the bound
\begin{equation}\label{delta1}
||y\cdot\nabla v||_{L^\infty}\lesssim C(K,K')e^{-\mu s}.
\end{equation}
\end{lemma}
\begin{proof}
\emph{$\mathbf{Step\ 1}.$}\ $L^\infty$ bound of $y\cdot\nabla v$ inside the ball.
We recall \eqref{yjs} that
$$
\partial_s{v}+\hat{L}v=\hat{F}
$$
where $\hat{L}v=-\Delta {v}+\hat{T}\cdot\nabla {v}+\hat{V}{v}$. See \eqref{yjs1} and \eqref{yjs2} for the precise definitions of $\hat{T}$, $\hat{V}$ and $\hat{F}$.

Take any $0<R<\infty$. We know from \eqref{largex}, \eqref{chca}, \eqref{1.9}, \eqref{2.0} and \eqref{semin} that
\begin{equation}\label{tidss1}
||\hat{T}(y)||_{L^\infty(B_{2R})}\lesssim1,\ ||\hat{V}(y)||_{L^\infty(B_{2R})}\lesssim1
 \end{equation}
 and
\begin{equation}\label{tidss2}
||\hat{F}||_{L^\infty(B_{2R})}\lesssim e^{-\mu s}.
\end{equation}
 We know from \eqref{chca} and \eqref{1.9} that
\begin{equation}\label{L2}
||v||_{L^2(B_{2R})}\lesssim e^{-\mu s}.
\end{equation}
By introducing a smooth localized function in time, combining \eqref{tidss1}, \eqref{tidss2} and \eqref{L2} and employing parabolic regularity, we obtain
\begin{equation}\label{tidss45}
||v(s)||_{W^{2,1;p}(B_{2R})}\lesssim  e^{-\mu s},\ s\in[s_0+\delta_0,s^*],
\end{equation}
where $1<p<\infty$ and $\delta_0$ is small, then by the Sobolev embedding and Schauder estimate we know that
\begin{equation}\label{tidssozx}
|| v(s)||_{C^2(B_R)}\lesssim  e^{-\mu s},\ s\in[s_0+2\delta_0,s^*].
\end{equation}
Hence, there exists a positive constant $\tilde{C}=\tilde{C}(K,K')$ such that
\begin{equation}\label{nh}
||y\cdot \nabla v(s)||_{L^\infty(B_R)}\le \frac{\tilde{C}}{2}Re^{-\mu s},\ s\in[s_0+2\delta_0,s^*].
\end{equation}
\emph{$\mathbf{Step\ 2}.$}\ $L^\infty$ bound of $y\cdot\nabla v$ outside the ball.
Set
$$
z=y\cdot\nabla v.
$$
We now use the commutator relation
$$
[\Delta,\Lambda]=2\Delta
$$
to find that
$$
z_s+\tilde{L} z=\tilde{F}
$$
where $\tilde{L} z=-\Delta z+\tilde{T}\cdot\nabla z+\tilde{V}z$ and
$$
\tilde{T}=\left[\frac{1}{2}-2\Phi_n-\left(\frac{\lambda_s}{\lambda}+\frac{1}{2}\right)-2v\right]y,
$$
$$
\tilde{V}=1-4y\cdot\nabla\Phi_n-12\Phi_n-2\left(\frac{\lambda_s}{\lambda}+\frac{1}{2}\right)-12v
$$
and
$$
\tilde{F}=\left(\frac{\lambda_s}{\lambda}+\frac{1}{2}\right)y\cdot\nabla\Lambda\Phi_n
+2z^2-\big[2\Delta v-2vy\cdot\nabla(y\cdot\nabla\Phi_n)-12y\cdot\nabla\Phi_n v\big]$$
We claim that
\begin{equation}\label{compaq}
 ||z(s)||_{L^\infty(|y|\ge R)}\le \tilde{C}Re^{-\mu s}\ \ s\in[s_0+2\delta_0,s^*].
\end{equation}
We next use parabolic comparison principle to prove the above estimate.
Let $\tilde{\psi}(y,s)=\tilde{C}Re^{-\mu s}$. By \eqref{nh}, on the lateral boundaty, we have
\begin{equation}\label{dfyuss}
\tilde{\psi}(s,R)\ge |z(s,R)|,\ \ s\in[s_0+2\delta_0,s^*],
\end{equation}
By \eqref{1.5}, and by continuity and differentiable of the solution, there holds
$$
 ||z(s_0+2\delta_0)||_{L^\infty}\le \tilde{C}Re^{-\mu s},
$$
for $K_0$ small enough. Hence
\begin{equation}\label{qdjj}
\tilde{\psi}(s_0+2\delta_0,y)\ge |z(s_0+2\delta_0,y)|,\ \ |y|\ge R.
\end{equation}
Since $\mu>0$ is small, then combining \eqref{largex}, \eqref{chca}, \eqref{1.9} and \eqref{2.0}, we have
\begin{equation}\label{compa}
\begin{aligned}
&\partial_s(\tilde{\psi}-z)-\Delta (\tilde{\psi}-z)+\tilde{T}\cdot\nabla (\tilde{\psi}-z)+\tilde{V}(\tilde{\psi}-z)=\tilde{C}Re^{-\mu s}(\tilde{V}-\mu)-\tilde{F}\ge0,
\end{aligned}
\end{equation}
for $R$ and $s_0$ large enough. By \eqref{dfyuss}, \eqref{qdjj} and \eqref{compa}, we employ the comparison principle to find that
$$
z\le \tilde{\psi}, \ \ \text{on}\ \{|y|\ge R\}\times [s_0+2\delta_0,s^*].
$$
We can also show similarly that
$$
-\tilde{\psi}\le z,
\ \ \text{on}\ \{|y|\ge R\}\times [s_0+2\delta_0,s^*],
$$
which proves the claim.
Combining \eqref{nh} and \eqref{compaq}, we get
\begin{equation}\label{cyt}
 ||z(s)||_{L^\infty}\le \tilde{C}Re^{-\mu s},\ \ s\in[s_0+2\delta_0,s^*].
\end{equation}
\emph{$\mathbf{Step\ 3}.$}\  Estimate in the time interval $[s_0,s_0+2\delta_0]$.  Combining \eqref{1.5}, then by continuity and differentiable of the solution, there holds
\begin{equation}\label{cyt1}
 ||z(s)||_{L^\infty}\le \tilde{C}Re^{-\mu s},\ \ [s_0,s_0+2\delta_0],
\end{equation}
for $K_0$ small enough.
We finalize the proof of this lemma by amalgamating \eqref{cyt} and \eqref{cyt1}.
%
\end{proof}

\subsection{Conclusion}
We are now in position to conclude the proof of Proposition \ref{bootstrap}.
\begin{proof}[Proof of Proposition \ref{bootstrap}]
Let us first recall the exit time
$$
s^*=\sup\{s\ge s_0\ \text{such that}\ \eqref{1.7}-\eqref{ball}\ \text{holds on}\ [s_0,s) \}.
$$
\emph{$\mathbf{Step\ 1}.$}\ Improved scaling control. We estimate from \eqref{chca},  \eqref{1.9} and \eqref{semin}:
\begin{equation}\label{fky}
\left|\frac{\lambda_s}{\lambda}+\frac{1}{2}\right|\lesssim K^2e^{-2\mu s},
\end{equation}
after integration we have
$$
\left|\log\left(\frac{\lambda(s)}{\lambda(s_0)}\right)+\frac{s-s_0}{2}\right|\lesssim K^2\int_{s_0}^\infty e^{-2\mu \tau}d\tau\lesssim1+o(1)
$$
for $s_0$ large enough. Then by \eqref{smallness0}, we obtain
\begin{equation}\nonumber
\lambda(s)=(\lambda(s_0)e^{\frac{s_0}{2}})e^{-\frac{s}{2}}(1+o(1))=e^{-\frac{s}{2}}(1+o(1))
\end{equation}
and hence
\begin{equation}\label{poaf}
\frac{e^{-\frac{s}{2}}}{2}\le\lambda(s)\le 2e^{-\frac{s}{2}}.
\end{equation}
\emph{$\mathbf{Step\ 2}.$}\ Improved weighted Sobolev bounds.\\
\emph{$L^2_\rho$ bound.} We know from \eqref{chca} and \eqref{Eys}  that
\begin{equation}\nonumber
\begin{aligned}
\frac{d}{ds}||\varepsilon||_{L^2_\rho}^2+c_n||\varepsilon||_{H^1_\rho}^2\lesssim e^{-4\mu s}\le e^{-2\mu s}
\end{aligned}
\end{equation}
for $ s_0$ large enough. We next define
\begin{equation}\label{parameter}
0<\mu\le\frac{c_n}{4},
\end{equation}
then
$$
\frac{d}{ds}||\varepsilon||_{L^2_\rho}^2+4\mu||\varepsilon||_{H^1_\rho}^2\le  e^{-2\mu s}.
$$
By Gronwall inequality and \eqref{1.5}, there holds
\begin{equation}\label{iiz}
\begin{aligned}
&||\varepsilon(s)||_{L^2_\rho}^2+2\mu e^{-2\mu(s-s_0))}\int_{s_0}^s||\varepsilon(\tau)||_{L^2_\rho}^2d\tau\\
&\le e^{-2\mu s} (e^{2\mu s_0}||\varepsilon(s_0)||_{L^2_\rho}^2)+ e^{-2\mu s}e^{2\mu s_0}\int_{s_0}^se^{-2\mu \tau}d\tau\lesssim K_0^2e^{-2\mu s}.
\end{aligned}
\end{equation}
\emph{$\mathbf{Step\ 3}$.}\ The Brouwer fixed point theorem. Combining the definitions of $s^*$ and the contradiction assumption \eqref{sdzx}, along with  \eqref{dfyusd}, \eqref{delta}, \eqref{delta1}, \eqref{poaf}, and \eqref{iiz}, then through a continuity argument and \eqref{chca}, we have
\begin{equation}\label{s*}
\sum_{j=2}^N|a_j(s^*)|^2=e^{-2\mu s^*}.
\end{equation}
From \eqref{chca}, \eqref{1.9} and \eqref{semin}, we obtain
\begin{equation}\nonumber
\begin{aligned}
\frac{1}{2}\frac{d}{ds}\sum_{j=2}^N|a_je^{\mu s}|^2&=\sum_{j=2}^Na_je^{2\mu s}[(\mu_j+\mu) a_j+((a_j)_s-\mu_ja_j)]\ge\mu\sum_{j=2}^N|a_j|^2e^{2\mu s}+O(K^2e^{-\mu s}),
\end{aligned}
\end{equation}
then there holds
\begin{equation}\label{fas}
\left(\frac{1}{2}\frac{d}{ds}\sum_{j=2}^N|a_je^{\mu s}|^2\right)(s^*)\ge \mu+O(K^2e^{-\mu s_0})>0
\end{equation}
for $s_0$ large enough, which means that the vector field is strictly outgoing. Let $\mathcal{B}$ denote the unit ball in $\mathbb{R}^{N-1}$.
We define the map $f:\mathcal{D}(f)\subset\mathcal{B}\to\partial\mathcal{B}$ by
$$
f:(a_j(0)e^{\mu s_0})_{2\le j\le N}\mapsto(a_{j}(s^*)e^{\mu s^*})_{2\le j\le N},
$$
where
$$
\mathcal{D}(f)=\{(a_j(0)e^{\mu s_0})_{2\le j\le N}: (a_{j,0}e^{\mu s_0})_{2\le j\le N}\in \mathcal{B} \text{ and  exit time $s^*$\ is\ finite}\}.
$$
If we take $(a_j(0)e^{\mu s_0})_{2\le j\le N}\in \partial\mathcal{B}$, then by \eqref{s*} and \eqref{fas}, there holds $\sum_{j=2}^N|a_j(s)|^2e^{2\mu s}>1$ for any $s>s_0$. Hence, $\partial {\mathcal{B}}\subset\mathcal{D}(f)$ and $f$ is the identity on $\partial {\mathcal{B}}$. Since the unstable modes $(a_j(s))_{2\le j\le N}$ depend continuouly with respect to the inital data of the unstable modes $(a_j(0))_{2\le j\le N}$, and  the vector field satisfies the strictly outgoing property \eqref{fas}, then $\mathcal{D}(f)$ is open and $f$ is continuous on $\mathcal{D}(f)$.

Since $f$
is continuous in $\mathcal{B}$, and the identity on the boundary, then this implies that $\partial {\mathcal{B}}$ is a retract of $\mathcal{B}$, but it contradicts to Brouwer's fixed theorem \cite{Brouwer}.
Therefore, we have $s^*=\infty$, which completes the proof of Proposition \ref{bootstrap}.
\end{proof}

Next we present the proof of Theorems \ref{theorem-1} and \ref{theorem} by proposition \ref{bootstrap}.
\begin{proof}[Proof of Theorems \ref{theorem-1} and \ref{theorem}]

The desired functions $(\varphi_j)_{1\leq j\leq N-1}$ are defined as follows. We take $R>0$ large enough and let, for $1\leq j\leq N-1$,
$$
\varphi_j(x)=\bar \phi_{j+1}(x),
$$
where $\bar \phi_{j}$ is defined in \eqref{id:def-bar-phij}. The desired set is then defined as $V_n=\textup{Span}(\varphi_1,...,\varphi_{N-1})^\perp$. We recall that $E_n$ is defined by \eqref{id:def-En}. Then any $v_0\in V_n$ can be uniquely decomposed as
$$
v_0=\bar u_0+\sum_2^N b_{j,0}(v_0)\bar \phi_j
$$
where $\bar u_0\in E_n$, and $b_{j,0}$ is a continuous linear function of $v_0$ for the $L^\infty$ topology. By \eqref{id:def-bar-psij}, the map $v_0\mapsto \bar u_0$ is an isomorphism from $V_n\cap L^\infty$ to $E_n$. We apply Proposition \ref{bootstrap} to $\bar u_0$, and use Lemma \ref{lem:change-decomposition}. This gives the existence of $(\bar a_{j,0}(v_0))_{2\leq j \leq N}$ such that, for the initial data
\begin{align} \label{id:initial-data-theorem-main}
u_0 & =\frac{1}{\bar \lambda_0^2}\left(U_n+\bar u_0+\sum_2^N \bar a_{j,0}(v_0)\bar \phi_j \right)\left(\frac{x}{\bar \lambda_0}\right) \\
\nonumber & =\frac{1}{\bar \lambda_0^2}\left(U_n+v_0+\sum_1^{N-1} (\bar a_{j+1,0}(v_0)-b_{j+1,0}(v_0))\varphi_j \right)\left(\frac{x}{\bar \lambda_0}\right) ,
\end{align}
the corresponding solution $u(t)$ satisfies the conclusion of Proposition \ref{bootstrap}. The desired functions $(c_j(v_0))_{1\leq j \leq N}$ in the statement of Theorem \ref{theorem} are then defined as
$$
c_{j}(v_0)=\bar a_{j+1,0}(\bar u_0)-b_{j+1,0}(v_0)
$$
for $1\leq j \leq N-1$. We will show in the next Subsection \ref{lip} that $a_j$ is a Lipschitz function of $\bar u_0\in E_n$, what will imply that $c_j$ is a Lipschitz function of $v_0\in V_n\cap L^\infty$. We will show below that the solution $u(t)$ satisfies all the desired properties $(\mathrm{i})$, $(\mathrm{ii})$ and $(\mathrm{iii})$ of Theorem \ref{theorem}. This will thus conclude its proof upon rescaling the initial data \ref{id:initial-data-theorem-main}.

Indeed, the radial initial data $u_0$ satisfies $ u_0(r)=6w_0(r)+2r\partial_rw_0(r) $ where $w_0$ is as in Proposition \ref{bootstrap}. The corresponding solution $u(t,x)$ of \eqref{1.1}
admits a decomposition
\begin{equation}\label{frsa}
u(t, x)=\frac{1}{\lambda(t)^2}\left(U_n+{\tilde{u}}\right)(s, y),\quad y=\frac{x}{\lambda(t)},
\end{equation}
on the time interval $[s_0,+\infty)$, and $\tilde{u}$ satisfies
\begin{equation}\label{redhat}
\tilde{u}(s,r)=6v(r)+2r\partial_rv(r).
\end{equation}

\emph{$\mathbf{Step\ 1}.$}\ Self similar blow up.
From \eqref{poaf}, we have
$$
T=\int_{s_0}^{+\infty}\lambda^2(s)ds\lesssim \int_{s_0}^{+\infty}e^{-s}ds<+\infty.
$$
Thus
\begin{equation}\label{times}
T-t=\int_{s}^{+\infty}\lambda^2(s)ds\sim e^{-s}.
\end{equation}
By \eqref{fky}, we get $ \left|\lambda\lambda_t+\frac{1}{2}\right|\lesssim (T-t)^{2\mu}$. Then integrating in $t$ and using $\lambda(T)=0$ yields
\begin{equation}\label{time0s}
\lambda(t)=\sqrt{(T-t)}(1+o(1)),\ \ t\to T.
\end{equation}
\emph{$\mathbf{Step\ 2.}$}\ Asymptotic stability of the self similar profile above scaling.
Since the eigenvectors $\psi_j$ are bounded for $j=2,\cdots N$, then by \eqref{chca} and \eqref{2.0} we have
$$
|v|\lesssim e^{-\mu s},
$$
then combining \eqref{ball}, \eqref{redhat} and \eqref{times}yield
\begin{equation}\label{frsv}
||\hat{u}(t)||_{L^\infty}\lesssim(T-t)^\mu\to 0,\ \ \text{as}\ t\to T,
\end{equation}
and \eqref{norms} is proved.\\
\emph{$\mathbf{Step\ 3.}$}\  Convergence of the solution.
We infer from \eqref{largex} that
\begin{equation}\label{frs}
U_n(x)\lesssim\frac{1}{|x|^2}.
\end{equation}
Combining \eqref{weihsg} and \eqref{redhat}, we obtain
\begin{equation}\label{moros}
|\tilde{u}(x)|\lesssim\frac{1}{|x|^2}+e^{-s},
\end{equation}
 hence by \eqref{frsa}, \eqref{times}, \eqref{time0s} and \eqref{frs}
\begin{equation}\label{moros1}
|{u}(x)|\lesssim\frac{1}{|x|^2}+1.
\end{equation}
For any $\delta_1>0$ and $|x|\ge \delta_1$,  by parabolic regularity and Ascoli-Arzela theorem, there exists a function $u^*$ such that we have
\begin{equation}\label{shzx12}
\lim_{t\to T}u(t,x)= u^*,
\end{equation}
and which satisfies by \eqref{moros1}
\begin{equation} \label{shzx12000}
|u^*(x)|\lesssim\frac{1}{|x|^2}+1.
\end{equation}
Assume $u(0)\in L^p$ for some $p\in[1,\frac{3}{2})$. Then $u(t)$ remains bounded in $L^p(\{|x|>1\})$ by standard parabolic regularity arguments. Hence $u^*\in L^p(\{|x|>1\})$. As $u^*\in L^p(\{|x|<1\})$ by \eqref{shzx12000} we deduce $u^*\in L^p(\mathbb R^3)$. From \eqref{moros1} and \eqref{shzx12} we get as $\delta_1\to0$,
\begin{equation}\nonumber
\begin{aligned}
\lim_{t\to T}||u(t)-u^*||_{L^p(\mathbb{R}^3)}^p&=\lim_{t\to T}\int_{|x|\le\delta_1}|u(t,x)-u^*(x)|^pdx\lesssim \int_0^{\delta_1} (r^{2-2p}+r^2)dr\to0,
\end{aligned}
\end{equation}
and \eqref{convergences} is proved.
\end{proof}
\subsection{The Lipschitz dependence}\label{lip}
The purpose of this subsection is to show the Lipschitz dependence of the set of the solution we constructed in this paper.

\begin{proposition}
Let $\varepsilon_0^{(1)}$ and $\varepsilon_0^{(2)}$ be two functions satisfying the assumptions of Proposition \ref{bootstrap} for two distinct functions $\bar u_0^{(1)}$ and $\bar u_0^{(2)}$ as in Lemma \ref{lem:initial-data}. Then the corresponding parameters $(a_j^{(1)}(0))_{2\le j\le N}$ and $(a_j^{(2)}(0))_{2\le j\le N}$ given by Proposition \ref{bootstrap} satisfy
\begin{equation} \label{bd:aj-lipschitz}
\sum_{j=2}^N|a_j^{(1)}(0)-a_j^{(2)}(0)|^2\lesssim||\varepsilon_0^{(1)}-\varepsilon_0^{(2)}||_{L^2_\rho}^2\lesssim \| \bar u_0^{(1)}-\bar u_0^{(2)}\|_{L^\infty}.
\end{equation}
Moreover, their respective blow-up times $T^{(1)}$ and $T^{(2)}$ satisfy
\begin{equation} \label{bd:T-lipschitz}
|T^{(1)}-T^{(2)}|\lesssim \| \bar u_0^{(1)}-\bar u_0^{(2)}\|_{L^\infty}.
\end{equation}
\end{proposition}

\begin{proof}
We denote $\varepsilon^{(i)}$, $\lambda^{(i)}$ and  $v^{(i)}$  the functions and parameter associated to $(\varepsilon^{(i)}_0,\lambda_0^{(i)})$. we next define the differences by
$$
\triangle \varepsilon:=\varepsilon^{(1)}-\varepsilon^{(2)},\ \triangle a_j:=a_j^{(1)}-a_j^{(2)},\ \triangle v:=v^{(1)}-v^{(2)}.
$$
At the same renormalized time $s$, the time evolution of difference satisfies
\begin{equation}\label{shzx}
\begin{aligned}
\triangle \varepsilon_s+L_n\triangle\varepsilon&=\left(\frac{\lambda^{(1)}_s}{\lambda^{(1)}}-\frac{\lambda^{(2)}_s}{\lambda^{(2)}}\right)\Lambda(\Phi_n+v^{(2)})+
\sum_{j=2}^N(\mu_j\triangle a_j-\triangle(a_j)_s)\psi_j\\
&\ \ +\left(\frac{\lambda^{(1)}_s}{\lambda^{(1)}}+\frac{1}{2}\right)\Lambda \triangle v+6(v^{(1)}+v^{(2)})\triangle v+\Lambda'((v^{(1)}+v^{(2)})\triangle v).
\end{aligned}
\end{equation}
\emph{$\mathbf{Step\ 1}$.} Modulation equations.
We claim that
\begin{equation}\label{shzxas}
\begin{aligned}
\left|\frac{\lambda^{(1)}_s}{\lambda^{(1)}}-\frac{\lambda^{(2)}_s}{\lambda^{(2)}}\right|+
\sum_{j=2}^N|\mu_j\triangle a_j-\triangle(a_j)_s)|\lesssim e^{-\mu s}\left(||\triangle\varepsilon||_{L^2_\rho}+\sum_{j=2}^N|\triangle a_j|\right).
\end{aligned}
\end{equation}
To prove the above estimate, we take the scalar product of \eqref{shzx} with $\psi_1=\frac{\Lambda \Phi_n}{||\Lambda \Phi_n||_{L^2_\rho}}$, then by
\eqref{orthogonality} and \eqref{nonbus}, we get
\begin{equation}\nonumber
\begin{aligned}
&\left(\frac{\lambda^{(1)}_s}{\lambda^{(1)}}-\frac{\lambda^{(2)}_s}{\lambda^{(2)}}\right)(\Lambda(\Phi_n+v^{(2)}),\psi_1)_\rho\\
&=-\left(\frac{\lambda^{(1)}_s}{\lambda^{(1)}}+\frac{1}{2}\right)(\Lambda \triangle v,\psi_1)_\rho
-6((v^{(1)}+v^{(2)})\triangle v,\psi_1)_\rho-(\Lambda'((v^{(1)}+v^{(2)})\triangle v),\psi_1)_\rho.
\end{aligned}
\end{equation}
By integrating by parts, we know from \eqref{taxiss}, \eqref{chca} and \eqref{2.0} that
\begin{equation}\label{delo3}
(\Lambda(\Phi_n+v^{(2)}),\psi_1)_\rho\lesssim1+O(e^{-\mu s}).
\end{equation}
From Lemma \ref{self-similar profile}, \eqref{taxiss} and \eqref{semin}, combining Cauchy-Schwarz,  we have
\begin{equation}\label{delo}
\begin{aligned}
\left|\left(\frac{\lambda^{(1)}_s}{\lambda^{(1)}}+\frac{1}{2}\right)(\Lambda \triangle v,\psi_1)_\rho\right|\lesssim e^{-2\mu s}\left(||\triangle\varepsilon||_{L^2_\rho}+\sum_{j=2}^N|\triangle a_j|\right).
\end{aligned}
\end{equation}
Since $||\psi_j||_{L^\infty}\lesssim 1$ for $j=2,\cdots, N$, combining \eqref{chca} and \eqref{2.0} yields
\begin{equation}\label{lili}
||v^{(i)}||_{L^\infty}\lesssim e^{-\mu s}\ \ i=1,2,
\end{equation}
then by \eqref{chca} and \eqref{2.0}, we obtain
\begin{equation}\label{delo1}
\begin{aligned}
\left|((v^{(1)}+v^{(2)})\triangle v,\psi_1)_\rho\right|\lesssim ||v^{(1)}+v^{(2)}||_{L^\infty}(\triangle v,\psi_1)_\rho\lesssim e^{-\mu s}\left(||\triangle\varepsilon||_{L^2_\rho}+\sum_{j=2}^N|\triangle a_j|\right),
\end{aligned}
\end{equation}
and similarly
\begin{equation}\label{delo2}
\begin{aligned}
\left|(\Lambda'((v^{(1)}+v^{(2)})\triangle v),\psi_1)_\rho\right|\lesssim e^{-\mu s}\left(||\triangle\varepsilon||_{L^2_\rho}+\sum_{j=2}^N|\triangle a_j|\right).
\end{aligned}
\end{equation}
Combining \eqref{delo3}, \eqref{delo}, \eqref{delo1} and \eqref{delo2} yields
$$
\left|\frac{\lambda^{(1)}_s}{\lambda^{(1)}}-\frac{\lambda^{(2)}_s}{\lambda^{(2)}}\right|\lesssim e^{-\mu s}\left(||\triangle\varepsilon||_{L^2_\rho}+\sum_{j=2}^N|\triangle a_j|\right).
$$
Taking the scalar product of \eqref{shzx} with $\psi_j$ for $j=2,\cdots, N$,  then by the same way, we have
$$
\sum_{j=2}^N|\mu_j\triangle a_j-\triangle(a_j)_s)|\lesssim e^{-\mu s}\left(||\triangle\varepsilon||_{L^2_\rho}+\sum_{j=2}^N|\triangle a_j|\right),
$$
which proves the claim.\\
\emph{$\mathbf{Step\ 2}$.} Localized energy estimate. We claim the following estimate
\begin{equation}\label{shzxlk0o}
\frac{d}{ds}||\triangle\varepsilon||_{L^2_\rho}^2+c_n||\triangle\varepsilon||_{L^2_\rho}^2\lesssim e^{-c\mu s}\sum_{j=2}^N|\triangle a_j|^2,
\end{equation}
for $0<c\le1$. We next prove the above estimate.
From \eqref{shzx} one obtains the identity
\begin{equation}\label{shzxlk}
\begin{aligned}
\frac{d}{ds}\frac{1}{2}||\triangle\varepsilon||^2_{L^2_\rho}&=-(L_n\triangle\varepsilon,\triangle\varepsilon)_\rho
+\left(\frac{\lambda^{(1)}_s}{\lambda^{(1)}}-\frac{\lambda^{(2)}_s}{\lambda^{(2)}}\right)(\Lambda (v^{(2)}),\triangle\varepsilon)_\rho\\
&\ \ +\left(\frac{\lambda^{(1)}_s}{\lambda^{(1)}}+\frac{1}{2}\right)(\Lambda \triangle v,\triangle\varepsilon)_\rho
+6((v^{(1)}+v^{(2)})\triangle v,\triangle\varepsilon)_\rho\\
&\ \ +(\Lambda'((v^{(1)}+v^{(2)})\triangle v),\triangle\varepsilon)_\rho.
\end{aligned}
\end{equation}
We know from the spectral gap \eqref{spectral gap} and the orthogonality condition \eqref{orthogonality} that
\begin{equation}\label{tehui0}
-(L_n\triangle\varepsilon,\triangle\varepsilon)_\rho\le -c_n||\triangle\varepsilon||^2_{H^1_\rho}.
\end{equation}
We next estimate all the other terms of \eqref{shzxlk}.
By \eqref{chca}, \eqref{1.9}, \eqref{lili} and \eqref{coercivity}, then combining integrating by parts and Cauchy-Schwarz, we get
\begin{equation}\label{tehui1}
\begin{aligned}
&\left(\frac{\lambda^{(1)}_s}{\lambda^{(1)}}-\frac{\lambda^{(2)}_s}{\lambda^{(2)}}\right)(\Lambda (v^{(2)}),\triangle\varepsilon)_\rho
\\
&\lesssim \left|\frac{\lambda^{(1)}_s}{\lambda^{(1)}}-\frac{\lambda^{(2)}_s}{\lambda^{(2)}}\right|\left[||v^{(2)}||_{L^2_\rho}||\triangle\varepsilon||_{L^2_\rho}
+\left|\int_{\mathbb{R}^5}v^{(2)}\nabla\cdot(y\triangle\varepsilon\rho)dy\right|\right]\\
&\lesssim\left|\frac{\lambda^{(1)}_s}{\lambda^{(1)}}-\frac{\lambda^{(2)}_s}{\lambda^{(2)}}\right|
\left[(||\varepsilon||_{L^2_\rho}+\sum_{j=2}^N|a_j^{(2)}|)||\triangle\varepsilon||_{L^2_\rho}
+||v^{(2)}||_{L^\infty}\left|\int_{\mathbb{R}^5}\nabla\cdot(y\triangle\varepsilon\rho)dy\right|\right]\\
&\lesssim \left|\frac{\lambda^{(1)}_s}{\lambda^{(1)}}-\frac{\lambda^{(2)}_s}{\lambda^{(2)}}\right|
\left[e^{-\mu s}||\triangle\varepsilon||_{L^2_\rho}+e^{-\mu s}||\triangle\varepsilon||_{H^1_\rho}\right]\\
&\lesssim e^{-2\mu s}\left(||\triangle\varepsilon||_{L^2_\rho}+\sum_{j=2}^N|\triangle a_j|\right)||\triangle\varepsilon||_{H^1_\rho}\lesssim e^{-2\mu s}\left(||\triangle\varepsilon||^2_{H^1_\rho}+\sum_{j=2}^N|\triangle a_j|^2\right).
\end{aligned}
\end{equation}
Similarly
\begin{equation}\label{tehui2}
\begin{aligned}
\left(\frac{\lambda^{(1)}_s}{\lambda^{(1)}}+\frac{1}{2}\right)(\Lambda \triangle v,\triangle\varepsilon)_\rho\lesssim e^{-2\mu s}\left(||\triangle\varepsilon||^2_{H^1_\rho}+\sum_{j=2}^N|\triangle a_j|^2\right).
\end{aligned}
\end{equation}
From Cauchy-Schwarz and Cauchy inequality, we deduce from \eqref{lili} that
\begin{equation}\label{tehui3}
\begin{aligned}
6((v^{(1)}+v^{(2)})\triangle v,\triangle\varepsilon)_\rho&\lesssim (||v^{(1)}||_{L^\infty}+||v^{(2)}||_{L^\infty})||\triangle v||_{L^2_\rho}||\triangle\varepsilon||_{L^2_\rho}\\
&\lesssim e^{-\mu s}\left(||\triangle\varepsilon||^2_{H^1_\rho}+\sum_{j=2}^N|\triangle a_j|^2\right).
\end{aligned}
\end{equation}
We integrate by parts and use Cauchy inequality and \eqref{coercivity} to estimate
\begin{equation}\label{tehui4}
\begin{aligned}
(\Lambda'((v^{(1)}+v^{(2)})\triangle v),\triangle\varepsilon)_\rho &
\lesssim (||v^{(1)}||_{L^\infty}+||v^{(2)}||_{L^\infty})\left|\int_{\mathbb{R}^5}\triangle v\nabla\cdot(y\triangle\varepsilon\rho)dy\right|\\
&\lesssim e^{-\mu s}\left(||\triangle\varepsilon||^2_{H^1_\rho}+\sum_{j=2}^N|\triangle a_j|^2\right).
\end{aligned}
\end{equation}
Combining \eqref{tehui0}, \eqref{tehui1}, \eqref{tehui2}, \eqref{tehui3} and \eqref{tehui4} yields
$$
\frac{d}{ds}||\triangle\varepsilon||_{L^2_\rho}^2+c_n||\triangle\varepsilon||_{H^1_\rho}^2\lesssim e^{-\mu s}\sum_{j=2}^N|\triangle a_j|^2\lesssim e^{-c\mu s}\sum_{j=2}^N|\triangle a_j|^2,
$$
for $s_0$ large enough and $0<c\le 1$, which proves our claim.\\
\emph{$\mathbf{Step\ 3}$.} Lipschitz bound by reintegration. We define
\begin{equation} \label{id:M-tildeM}
M:=\sup_{s\ge s_0}\sum_{j=2}^N|\triangle a_j(s)|e^{\mu s}<+\infty,\ \ \tilde{M}:=\sup_{s\ge s_0}||\triangle\varepsilon||^2_{L^2_\rho}e^{2\mu s}<+\infty,
\end{equation}
which are finite by \eqref{chca} and \eqref{1.9}. For $j=2,\cdots,N$, we reintegrate \eqref{shzxas} to find
\begin{equation}\label{sozc}
\begin{aligned}
\triangle a_j&=\triangle a_j(0)e^{\mu_j(s-s_0)}+e^{\mu_js}\int_{s_0}^se^{-\mu_j\tau}O(e^{-\mu\tau}(||\triangle\varepsilon||_{L^2_\rho}+\sum_{j=2}^N|\triangle a_j|))d\tau\\
&=e^{\mu_js}\left(\triangle a_j(0)e^{-\mu_js_0}+\int_{s_0}^sO(e^{-(\mu_j+2\mu)\tau}(M+\sqrt{\tilde{M}}))d\tau\right)\\
&=e^{\mu_js}\left(\triangle a_j(0)e^{-\mu_js_0}+\int_{s_0}^\infty O(e^{-(\mu_j+2\mu)\tau}(M+\sqrt{\tilde{M}}))d\tau\right)\\
&\ \ \ -e^{\mu_js}\int_{s}^\infty O(e^{-(\mu_j+2\mu)\tau}(M+\sqrt{\tilde{M}}))d\tau.
\end{aligned}
\end{equation}
The integral in the above is convergent and satisfies
$$
\left|\int_{s}^\infty O(e^{-(\mu_j+2\mu)\tau}(M+\sqrt{\tilde{M}}))d\tau\right|\lesssim e^{-(\mu_j+2\mu)s}(M+\sqrt{\tilde{M}}).
$$
As we know from \eqref{chca} that $|\triangle a_j|\lesssim e^{-\mu s}$, hence the parameter in front of the diverging term $e^{\mu_js}$ must be zero, i.e.
$$
\triangle a_j(0)e^{-\mu_js_0}+\int_{s_0}^\infty O(e^{-(\mu_j+2\mu)\tau}(M+\sqrt{\tilde{M}}))d\tau=0,
$$
which means that
\begin{equation}\label{Haduro}
|\triangle a_j(0)|\lesssim e^{-2\mu s_0}\left(M+\sqrt{\tilde{M}}\right),
\end{equation}
and from \eqref{sozc} one obtains $|\triangle a_j(s)|\lesssim e^{-2\mu s}(M+\sqrt{\tilde{M}})$. Then by the definition of $M$,
\begin{equation}\label{Hadur}
M\lesssim e^{-\mu s_0}\sqrt{\tilde{M}},
\end{equation}
for $s\ge s_0$ large enough.
We reintegrate the estimate \eqref{shzxlk0o}, by $0<c\le 1$ and $4\mu\le c_n$:
\begin{equation}\nonumber
\begin{aligned}
||\triangle\varepsilon||^2_{L^2_\rho}&\lesssim||\triangle\varepsilon(0)||^2_{L^2_\rho}e^{-c_n(s-s_0)}+e^{-c_ns}\int_{s_0}^se^{c_n\tau}e^{-c\mu\tau}\sum_{j=2}^N
|\triangle a_j|^2d\tau\\
&\lesssim ||\triangle\varepsilon(0)||^2_{L^2_\rho}e^{-c_n(s-s_0)}+e^{-(c+2)\mu s}M^2,
\end{aligned}
\end{equation}
then combining \eqref{Hadur}, $0<c\le 1$ and $4\mu\le c_n$, we have
\begin{equation}\label{coruS}
\tilde{M}\lesssim e^{c_n s_0}||\triangle\varepsilon(0)||^2_{L^2_\rho},
\end{equation}
for $s\ge s_0$ large enough.
From \eqref{Hadur} one obtains
\begin{equation}\label{coruS1}
M\lesssim e^{\mu s_0}||\triangle\varepsilon(0)||_{L^2_\rho}.
\end{equation}
Injecting the estimates \eqref{coruS} and \eqref{coruS1} into \eqref{Haduro} yields $
|\triangle a_j(0)|\lesssim ||\triangle\varepsilon(0)||_{L^2_\rho}$. This shows the first inequality in \eqref{bd:aj-lipschitz}, while the second comes from applying Lemma \ref{lem:initial-data}.

\medskip

\noindent \textbf{Step 4}. \emph{Lipschitz dependance of the blow-up time}. Injecting \eqref{coruS} and \eqref{coruS1} in \eqref{id:M-tildeM} one finds that
\begin{equation} \label{bd:aj-varepsilon-Lipschitz-s}
\sum_{j=2}^N |\triangle a_j(s)|+\| \triangle \varepsilon(s)\|_{L^2_\rho}\lesssim e^{-\mu(s-s_0)}\| \triangle \varepsilon(0)\|_{L^2_\rho}.
\end{equation}
We next prove \eqref{bd:T-lipschitz}. Injecting \eqref{bd:aj-varepsilon-Lipschitz-s} back in the modulation equation \eqref{shzxas} shows
$$
\left| \frac{d}{ds} \log \left(\frac{\lambda^{(1)}(s)}{\lambda^{(2)}(s)}\right)\right|\lesssim e^{-\mu(2s-s_0)}\| \triangle \varepsilon(0)\|_{L^2_\rho}.
$$
Integrating over time, this shows
$$
\left|  \log \left(\frac{\lambda^{(1)}(s)}{\lambda^{(2)}(s)}\right)-\log \left(\frac{\lambda^{(1)}(s_0)}{\lambda^{(2)}(s_0)}\right)\right|\lesssim e^{-\mu s_0} \| \triangle \varepsilon(0)\|_{L^2_\rho}.
$$
Using the Taylor expansion $\log(1+x)=x+O(x^2)$ as $x\to 0$ in the above inequality, then \eqref{bd:aj-lipschitz} and $\left| \frac{\lambda^{(1)}(s_0)}{\lambda^{(2)}(s_0)}-1\right|\lesssim \| \bar u_0^{(1)}-\bar u_0^{(2)}\|_{L^2}$ by Lemma \ref{lem:initial-data} and \eqref{bd:aj-lipschitz}, we have
\begin{align}\nonumber
\lambda^{(2)}(s) & =\lambda^{(1)}(s)\left(1+O\left(\left| \frac{\lambda^{(1)}(s_0)}{\lambda^{(2)}(s_0)}-1\right|+e^{-\mu s_0} \| \triangle \varepsilon(0)\|_{L^2_\rho} \right) \right) \\
 \label{bd:T-Lipschitz-inter} &=  \lambda^{(1)}(s)\left(1+O( \| \bar u_0^{(1)}-\bar u_0^{(2)}\|_{L^\infty})\right)
\end{align}
Since $\frac{dt}{ds}=\lambda^2$, the two blow-up times of the solutions are $T^{(1)}=\int_{s_0}^{\infty}(\lambda^{(1)}(s))^2 ds$ and $T^{(2)}=\int_{s_0}^{\infty}(\lambda^{(2)}(s))^2 ds$ respectively. Injecting \eqref{bd:T-Lipschitz-inter} in these two identities we have
\begin{align*}
\left|T^{(1)}-T^{(2)}\right| & = \left| \int_{s_0}^\infty [(\lambda^{(1)}(s))^2 -(\lambda^{(2)}(s))^2] ds\right| \\
&\quad \lesssim \int_{s_0}^\infty (\lambda^{(1)}(s))^2 \| \bar u_0^{(1)}-\bar u_0^{(2)}\|_{L^\infty} ds \ \lesssim  T^{(1)} \| \bar u_0^{(1)}-\bar u_0^{(2)}\|_{L^\infty}.
\end{align*}
This shows the desired Lipschitz dependance \eqref{bd:T-lipschitz}.
\end{proof}
\section{Appendix}
\setcounter{equation}{0}
\renewcommand\theequation{A.\arabic{equation}}

\begin{lem}\label{weighted estimate}
Let $u$ radial with $u$, $\partial_r u\in L^2_\rho(\mathbb{R}^5)$, then we have
\begin{equation}\label{coercivity}
||ru||_{L^2_\rho}\le ||u||_{H^1_\rho}.
\end{equation}
\end{lem}
\begin{proof}
We refer to the proof of Lemma A.1. in \cite{Collot-Pierre-2019} which proves the analogue of \eqref{coercivity} in dimension $3$. The arguments apply directly to the present case to show \eqref{coercivity}.
\end{proof}

\section*{Acknowledgements}

This result is part of the ERC starting grant project FloWAS that has received funding from the European Research Council (ERC) under the European Union's Horizon 2020 research and innovation programme (Grant agreement No. 101117820). C. Collot is supported by the CY Initiative of Excellence Grant "Investissements dAvenir" ANR-16-IDEX-0008 via Labex MME-DII, by the ANR grant "Chaire Professeur Junior" ANR-22-CPJ2-0018-01. K. Zhang is supported by China Scholarship Council (No.202206460045).

\section*{Data Availability Statement}
Data sharing is not applicable to this article as no datasets were generated or analysed during the current study.

\end{document}